\documentclass[12pt]{amsart}

\usepackage{color}
\theoremstyle{plain}
\newtheorem{theorem}{Theorem}[section]
\newtheorem{proposition}[theorem]{Proposition}
\newtheorem{lemma}[theorem]{Lemma}

\newtheorem*{signconjecture}{Euler Characteristic Sign Conjecture}
\newtheorem{definition}[theorem]{Definition}
\theoremstyle{remark}
\newtheorem{remark}{Remark}
\numberwithin{equation}{section}
\DeclareMathOperator{\lk}{lk}
\DeclareMathOperator{\conv}{conv}
\DeclareMathOperator{\sech}{sech}
\newcommand{\scc}{^{(sc)}}

\newcommand{\K}{\mathcal{K}}
\newcommand{\G}{\mathcal{G}}

\title{The combinatorics of hyperbolized manifolds}
\author{Allan L. Edmonds}
\address{Department of Mathematics, Indiana University, Bloomington, IN 47405}
\email{edmonds@indiana.edu}

\author{Steven Klee}
\address{Department of Mathematics, Seattle University, Seattle, WA 98122}
\email{klees@seattleu.edu}

\date{\today}

\begin{document}
\maketitle

\begin{abstract}
A topological version of a longstanding conjecture of H. Hopf, originally proposed by W. Thurston, states that the sign of the Euler characteristic of a closed aspherical manifold of dimension $d=2m$ depends only on the parity of $m$. Gromov defined several hyperbolization functors which produce an aspherical manifold from a given simplicial or cubical manifold. We investigate the combinatorics of several of these hyperbolizations and verify the Euler Characteristic Sign Conjecture for each of them. In addition, we explore further combinatorial properties of these hyperbolizations as they relate to several well-studied generating functions.
\end{abstract}

\tableofcontents

\section{Introduction}\label{sec:intro}

We study certain combinatorial aspects of the following fundamental unsolved problem in geometric topology, going back to a conjecture of H. Hopf about Riemannian manifolds of {nonpositive} curvature, and first formulated in dimension 4, and indeed all higher even dimensions, as a question by W. Thurston as early as 1977. (See Problem 4.10 in the Kirby Problem List \cite{Kirby1997}.)
\begin{signconjecture}
If $M$ is a closed, aspherical manifold of dimension $d = 2m$, then the Euler characteristic of $M$ satisfies $$(-1)^m \chi(M) \geq 0.$$
\end{signconjecture}
\noindent
Recall that a manifold $M$ is said to be \emph{aspherical} if $\pi_i(M) = 0$ for all $i \geq 2$ or, equivalently, the universal covering space of $M$ is contractible.  The conjectured sign corresponds to the sign of the Euler characteristic of a product of $m$ surfaces of genus $g \ge 1$.

There are certain ``hyperbolization functors,'' originally due to Gromov \cite{Gromov1987}, that assign to any simplicial or cubical $d$-manifold $\K$ an aspherical $d$-manifold $\mathcal{H}(\K)$.  We refer to Charney and Davis \cite{CharneyDavis1995b}; Davis and Januszkiewicz\cite{DavisJanuszkiewicz1991}; Davis, Januszkiewicz, and Weinberger\cite{DavisJanuszkiewicWeinberger2001}; and Paulin \cite{Paulin1991} for more details and analysis of such hyperbolization procedures. We will analyze the combinatorics of two  such hyperbolizations -- the M\"obius band hyperbolization and a somewhat more subtle one that we name the Gromov hyperbolization -- for both simplicial and cubical complexes and show that the Sign Conjecture is satisfied for each of them   (see Theorems \ref{thm:cubical-mobius}, \ref{thm:simplicial-mobius}, \ref{thm:cubical-gromov}, and \ref{thm:simplicial-gromov}).

\begin{theorem} \label{thm:main-thm}
If $\K$ is any closed simplicial or cubical manifold of dimension $d=2m$, and $\mathcal{H}$ denotes either the M\"obius band hyperbolization or Gromov hyperbolization, then $$(-1)^m \chi(\mathcal{H}(\K)) \geq 0.$$
\end{theorem}

When $\K$ is a cubical $4$-manifold, we have the following explicit formula for the M\"obius band hyperbolization in terms of the face numbers of the underlying cubical manifold,
$$
\chi(\mathcal{H}(\K))=f_{0}(\K)-f_{1}(\K)+2f_{3}(\K)\geq 0.
$$
This  inequality was originally proved by Janusziewicz \cite{Januszkiewicz1991} by directly studying chains of faces in the non-positively curved complex $\mathcal{H}(\K)$. We generalize this approach to prove the result in all even dimensions.  Before proceeding to the definitions and background material that will be necessary for the remainder of this paper, we will outline the general approach to our proof of the various manifestations of Theorem \ref{thm:main-thm}.  For simplicity, we will outline our approach for cubical complexes here; however, we will later see that the same approach will work for simplicial complexes as well.  

We apply one of the aforementioned hyperbolization functors to a given $d$-dimensional cubical complex $\K$ to obtain a hyperbolized complex $\mathcal{H}(\K)$.  Each of these hyperbolization functors is described by an inductive process in which the faces of $\K$ are replaced with certain hyperbolized cells, beginning with the $2$-dimensional faces and inducting to the $d$-dimensional faces.  Because of this, we are able to express the Euler characteristic of $\mathcal{H}(\K)$ as a function of the number of faces in the original complex $\K$.  Specifically, we define certain \textit{hyperbolization coefficients}, $a_{\mathcal{H}}(k)$ for $k=0,1,2,\ldots,$ so that 
\begin{equation} \label{hyperbolized-euler-char-general}
\chi(\mathcal{H}(\K)) = \sum_{k=0}^d a_{\mathcal{H}}(k)f_k(\K),
\end{equation}
where $f_k(\K)$ denotes the number of $k$-dimensional faces in the original complex $\K$. 

The face numbers (or $f$-numbers) $f_k(\K)$ are  natural invariants of $\K$ from the perspective of computing Euler characteristics; however, their behavior is rather fickle from a combinatorial perspective.  To address this problem, we will apply a standard combinatorial transformation which gives an equivalent family of invariants known as the \textit{short cubical $h$-numbers} of $\K$, which are denoted by $h\scc(\K)$.  Combinatorially, the short cubical $h$-numbers of a closed cubical manifold are well-behaved: for example, they are known to be symmetric by Klee's Dehn-Sommerville equations \cite{Klee1964}, and they are known to be nonnegative by a result of Stanley \cite{Stanley1975}.  Since the short cubical $h$-numbers are defined as a certain integer combination of the $f$-numbers, we can express the Euler characteristic of $\mathcal{H}(\K)$ as 
\begin{equation} \label{new-euler-char-formula}
\chi(\mathcal{H}(\K)) = \sum_{j=0}^dc_{\mathcal{H}}(j,d)h\scc_j(\K),
\end{equation}
for some other family of constants $c_{\mathcal{H}}(j,d)$ that are also rational combinations of the hyperbolization coefficients $a_{\mathcal{H}}(k)$ (but depend on the dimension of the complex in question).  

Just as the transformation from $f$-numbers to short cubical $h$-numbers may seem to be unwarranted, this transformation of hyperbolization coefficients into $c_{\mathcal{H}}(j,d)$ coefficients may seem somewhat arbitrary.  However, we are able to show that these new coefficients are beautifully structured: 
\begin{itemize}
\item$ c_{\mathcal{H}}(j,d) = (-1)^dc_{\mathcal{H}}(d-j,d)$,
\item $c_{\mathcal{H}}(j,d) \geq 0$ for all $j$ when $d \equiv 0 \mod 4$, and
\item $c_{\mathcal{H}}(j,d) \leq 0$ for all $j$ when $d \equiv 2 \mod 4$.  
\end{itemize}
Since the short cubical $h$-numbers of a closed cubical manifold are nonnegative, the fact that the coefficients $c_{\mathcal{H}}(j,d)$ are nonnegative/nonpositive when $d$ is even, together with the expression of $\chi(\mathcal{H}(\K))$ in Equation \eqref{new-euler-char-formula}, immediately implies Theorem \ref{thm:main-thm}.   


The conclusion of Theorem \ref{thm:main-thm} has in principle been known to experts for some time. We sketch an outline communicated to us by M. Davis and T. Januszkiewicz. The constructions in question involve barycentric subdivisions. The local contributions that build up the Euler characteristic are given by the Charney-Davis \cite{CharneyDavis1995a} quantity associated with the flag triangulated $(d-1)$-sphere links of the vertices. Because of the barycentric subdivisions, results of R. Stanley \cite{Stanley1994} and Karu  \cite{Karu2006} show that these local contributions are nonnegative/nonpositive as required. The virtue of the present approach lies in its more explicit and more elementary derivation and the intriguing combinatorial connections.


The organization of the remainder of this paper is as follows. In Section \ref{sec:prelims} we give a brief introduction to the combinatorics of $f$-numbers and $h$-numbers.  In Section \ref{sec:define-hyperbolizations}, we will define the M\"obius and Gromov hyperbolizations of a cubical or simplicial manifold.  In Section \ref{sec:mobius} we study the combinatorics of the M\"obius band hyperbolization, and in Section \ref{sec:gromov} we study the combinatorics of the Gromov hyperbolization.  In each if these sections, we will compute the Euler characteristic of a hyperbolized complex as in Equation \eqref{hyperbolized-euler-char-general} and study the properties of the corresponding hyperbolization coefficients.  We will conclude in Section \ref{sec:gen-funs} by studying interesting generating functions that arise from the hyperbolization coefficients.

\section{Background on combinatorial geometry} \label{sec:prelims}

In this section, we will define simplicial and cubical complexes combinatorially and discuss their geometric and combinatorial properties. For further details on these complexes, see \cite{Stanley-CCA}. The short cubical $h$-vector of Adin \cite{Adin1996} and its simplicial analogue due to Hersh and Novik \cite{HershNovik2002} will be central to our studies.

\subsection{Simplicial complexes and manifolds}

An \textit{(abstract) simplicial complex} $\Delta$ on the (finite) vertex set $V = V(\Delta)$ is a collection of subsets $F \subseteq V$ (called \textit{faces}) with the property that if $F \in \Delta$ and $G \subseteq F$, then $G \in \Delta$.  If $|V(\Delta)| = n$, there is a natural \textit{geometric realization} of $\Delta$, denoted $||\Delta|| \subseteq \mathbb{R}^n$ which identifies each face $F = \{i_1,\ldots,i_r\} \in \Delta$ with $||F|| = \conv\{e_{i_1}, \ldots, e_{i_r}\}$, where $e_j$ denotes the $j$-th standard basis vector.  

The \textit{dimension} of a face $F \in \Delta$ is $\dim F = |F|-1$, and the dimension of $\Delta$ is $\max\{\dim F: F \in \Delta\}$.  A simplicial complex $\Delta$ is \textit{pure} of all of its facets (maximal faces under inclusion) have the same dimension.  

The \textit{link} of a face $F$ in the simplicial complex $\Delta$ is the subcomplex $$\lk_{\Delta}(F):=\{G \in \Delta: F \cap G = \emptyset, F \cup G \in \Delta\}.$$  Note that when $\Delta$ is pure, $\lk_{\Delta}(F)$ is pure of dimension $\dim(\Delta) - |F|$ for any face $F \in \Delta$.  The following result of Munkres shows the importance of links in defining simplicial manifolds. 

\begin{lemma}{\rm{(Munkres \cite[Lemma 3.3]{Munkres1984})}} 
Let $\Delta$ be a simplicial complex, and let $F$ be a nonempty face of $\Delta$.  If $p$ is a point in the relative interior of $||F|| \subseteq ||\Delta||$, then $$H_i(||\Delta||, ||\Delta||-p) = \widetilde{H}_{i-|F|}(\lk_{\Delta}(F)).$$
\end{lemma}

When $||\Delta||$ is a $d$-dimensional topological manifold (without boundary), we know that the relative homology groups of the pair $(||\Delta||, ||\Delta||-p)$ are isomorphic to those of a $(d-1)$-sphere.  We use this to motivate the following definition. 

\begin{definition}
Let $\mathbf{k}$ be a field or the ring of integers, and let $\Delta$ be a $d$-dimensional simplicial complex.  We say that $\Delta$ is a $\mathbf{k}$-\textit{homology manifold} if 
\begin{displaymath}
\widetilde{H}_i(\lk_{\Delta}(F);\mathbf{k}) \cong 
\begin{cases}
\mathbf{k} & \text{ if $i = d-|F|$,} \\
0 & \text{otherwise,}
\end{cases}
\end{displaymath}
for all nonempty faces $F \in \Delta$.  If, in addition, $||\Delta||$ has the homology of $\mathbb{S}^{d}$ over $\mathbf{k}$, we say that $\Delta$ is a $\mathbf{k}$-\textit{homology sphere}.
\end{definition}

Let $\Delta$ be a $d$-dimensional simplicial complex.  We define the $f$-numbers of $\Delta$ as $f_{i}(\Delta) := \#\{F \in \Delta: \dim F = i\}$ for $-1 \leq i \leq d$.  It is often more convenient to study a certain integer transformation of the $f$-numbers called the $h$-numbers of $\Delta$, which are defined by $$h_j(\Delta):=\sum_{i=0}^j(-1)^{j-i}{d+1-i \choose d+1-j}f_{i-1}(\Delta),$$ for all $0 \leq j \leq d+1$.  One can check that the $f$-numbers of $\Delta$ are related to its $h$-numbers by $f_{i-1}(\Delta) = \sum_{j=0}^i{d+1-j \choose d+1-i}h_j(\Delta)$ so knowing the $f$-numbers of $\Delta$ is equivalent to knowing its $h$-numbers. 

In general, there is no reason to expect the $h$-numbers of a simplicial complex to be nonnegative (and in fact, in general they are not); however, Stanley \cite{Stanley1975} showed that the $h$-numbers of a certain family of simplicial complexes known as \textit{Cohen-Macaulay} complexes are nonnegative.  We will avoid giving the precise algebraic definition of a Cohen-Macaulay simplicial complex here because we will use the following theorem of Reisner to give a topological definition that is more suitable for this paper. 

\begin{theorem}{\rm{(Reisner's criterion \cite[Theorem 1]{Reisner1976})}}\label{Reisners-criterion}
Let $\Delta$ be a $d$-dimensional simplicial complex and let $\mathbf{k}$ be a field.  Then $\Delta$ is Cohen-Macaulay over $\mathbf{k}$ if and only if $\widetilde{H}_i(\lk_{\Delta}(F);\mathbf{k}) = 0$ for all $i < d-|F|$ and all faces $F \in \Delta$ (including $F = \emptyset$).
\end{theorem}

In particular, a $\mathbf{k}$-homology sphere is Cohen-Macaulay over $\mathbf{k}$.  By studying an associated structure known as the face ring of a simplicial complex, Stanley proved that the $h$-numbers of Cohen-Macaulay complexes are nonnegative.

\begin{theorem}{\cite[Corollary 4.3]{Stanley1975}}\label{h-nums-nonnegative}
Let $\Delta$ be a Cohen-Macaulay simplicial complex of dimension $d$.  Then $h_j(\Delta) \geq 0$ for all $0 \leq j \leq d+1$. 
\end{theorem}

Now suppose $\Delta$ is a $\mathbf{k}$-homology manifold.  By Reisner's criterion, the link of any vertex $v \in \Delta$ is Cohen-Macaulay over $\mathbf{k}$, and hence the $h$-numbers of $\lk_{\Delta}(v)$ are nonnegative by Theorem~\ref{h-nums-nonnegative}.  Hersh and Novik \cite{HershNovik2002} defined the \textit{short simplicial $h$-numbers} of a simplicial complex by $$\widetilde{h}_j(\Delta):=\sum_{v \in V(\Delta)}h_j(\lk_{\Delta}(v)).$$  It follows that $\widetilde{h}_j(\Delta) \geq 0$ for all $j$ when $\Delta$ is a simplicial (homology) manifold.  Just as the $f$-numbers of a simplicial complex can be written as nonnegative integer combinations of its $h$-numbers, the following result shows that the $f$-numbers can be expressed in terms of the short simplicial $h$-numbers as well.

\begin{proposition}{\rm{(\cite[Lemma 1(i)]{HershNovik2002})}}
Let $\Delta$ be a pure $d$-dimensional simplicial complex.  Then 
\begin{equation}\label{simplicial:f-to-short-h}
(i+1)f_i(\Delta) = \sum_{j=0}^i{d-j \choose d-i}\widetilde{h}_j(\Delta),
\end{equation}
for all $0 \leq i \leq d$.
\end{proposition}

\subsection{Cubical complexes}
Let $I^d$ denote the standard cube $[-1,1]^d$ in $\mathbb{R}^d$.  A \textit{cubical complex} $\mathcal{K}$ on the (finite) vertex set $V = V(\mathcal{K})$ is a collection of subsets of $V$, partially ordered by inclusion, satisfying the following properties: 
\begin{enumerate}
\item $\mathcal{K}$ has a minimal element, usually denoted $\emptyset$.
\item For all $v \in V$, the singleton $\{v\} \in \mathcal{K}$.
\item For any nonempty $F \in \mathcal{K}$, the interval $[\emptyset, F] = \{G \in \mathcal{K}: \emptyset \subseteq G \subseteq F\}$ is isomorphic to the face poset of a cube of some dimension.
\item If $F,F' \in \mathcal{K}$, then $F \cap F'$ is a face of $\mathcal{K}$. \label{lattice-condition}
\end{enumerate}

Once again, the elements $F \in \mathcal{K}$ are called faces.  If $[\emptyset, F]$ is isomorphic to the face poset of $I^i$, then we say $F$ is a $i$-dimensional face of $\mathcal{K}$.  This makes $\mathcal{K}$ a graded poset if we declare that an $i$-dimensional face of $\mathcal{K}$ has rank $i+1$; and, in fact, a graded lattice by condition (\ref{lattice-condition}) of the definition of a cubical complex. 

The \textit{link} of a face $F \in \mathcal{K}$ is $\lk_{\mathcal{K}}(F) = \{G \in \mathcal{K}: G \supseteq F\}$, with minimal element $F$.  If $\mathcal{K}$ is a pure $d$-dimensional cubical complex and $F \in \mathcal{K}$ is a nonempty face of dimension $i$, then $\lk_{\mathcal{K}}(F)$ is a pure \textit{simplicial} complex of dimension $d-i-1$.  Once again, we say that a pure $d$-dimensional cubical complex $\mathcal{K}$ is a $\mathbf{k}$\textit{-homology manifold} if the link of any nonempty face $F \in \mathcal{K}$ is a $\mathbf{k}$-homology sphere of dimension $d-\dim F - 1$.  

Adin~\cite{Adin1996} defined the \textit{short cubical} $h$-numbers of a $d$-dimensional cubical complex $\mathcal{K}$ by $$h_j\scc(\mathcal{K}) = \sum_{v \in V(\mathcal{K})}h_j(\lk_{\mathcal{K}}(v)),$$ where $h_j(\lk_{\Delta}(v))$ is the ordinary simplicial $h$-number of the simplicial complex $\lk_{\mathcal{K}}(v)$. Despite our presentation here, Adin's short cubical $h$-vector historically preceded (and motivated) Hersh and Novik's short simplicial $h$-vector.

Once again, $h_j\scc(\mathcal{K}) \geq 0$ for all $0 \leq j \leq d$ when $\mathcal{K}$ is a cubical $\mathbf{k}$-homology manifold since the link of each of its vertices is Cohen-Macaulay over $\mathbf{k}$.  Further, the $f$-numbers of a cubical complex can be expressed as a nonnegative linear combination of its short cubical $h$-numbers. 

\begin{proposition}{\rm{(\cite[Lemma 1(iii)]{Adin1996})}}
Let $\mathcal{K}$ be a $d$-dimensional cubical complex.  Then $$2^if_i(\mathcal{K}) = \sum_{j=0}^i{d-j \choose d-i}h_j\scc(\mathcal{K}),$$ for all $0 \leq i \leq d$.
\end{proposition}

\subsection{Euler Characteristic}
The Euler characteristic of a finite cell complex $X$ is defined to be
\[
\chi(X) = \sum_{i}(-1)^if_{i}(X).
\]
We will use a few simple facts. For example, the product formula $\chi(X\times Y)=\chi(X)\chi(Y)$, and the sum formula $\chi(X\cup Y)=\chi(X)+\chi(Y)-\chi(X\cap Y)$. If $p:X\to Y$ is an $n$ to $1$ covering map, then $\chi(X)=n\cdot\chi(Y)$. The Euler characteristic is a homotopy invariant. If $X$ is a compact odd-dimensional manifold without boundary, then $\chi(X)=0$, which follows from Poincar\'e duality. Finally, if $X$ is a compact odd-dimensional manifold with boundary, then $\chi(X)=\frac12\chi(\partial X)$.
\subsection{Asphericity}
We will use throughout the following theorem of J.~H.~C.~Whitehead. Recall that a path-connected space is aspherical if its higher homotopy groups $\pi_{k}(X)=0$ for $k\ge 2$.
\begin{theorem}[Whitehead \cite{Whitehead1939}]
Suppose a cell complex $X$ can be expressed as the union of two aspherical subcomplexes $X_{1}$ and $X_{2}$ such that each component of the intersection $X_{12}$ is aspherical and the inclusion induced homomorphisms $\pi_{1}(X_{12},x_{0})\to \pi_{1}(X_1,x_{0})$ and $\pi_{1}(X_{12},x_{0})\to \pi_{1}(X_2,x_{0})$ are injective for all choices of base point $x_{0}\in X_{12}$. Then $X$ is aspherical.
\end{theorem}
The idea of the proof is that the universal covering of $X$ can be constructed from contractible pieces with pairwise intersections contractible. From this information one can conclude that the homology groups and higher homotopy groups are trivial. The result follows.

\section{Hyperbolization techniques} \label{sec:define-hyperbolizations}

In this section we will define the various hyperbolization techniques that will be studied in the remainder of the paper.  

\subsection{The M\"obius band hyperbolization}
The central idea behind  this technique is the following observation of Gromov \cite{Gromov1987}. 

\begin{proposition}
Let $M$ be an aspherical $d$-manifold without boundary that admits a fixed-point free involution.  Then there is an aspherical $(d+1)$-manifold $W$ with $\partial W = M$, such that $\pi_{1}(M)\to \pi_{1}(W)$ is injective.  
\end{proposition}
\begin{proof}
Let $\varphi: M \rightarrow M$ be a fixed-point free involution on $M$.  We can extend $\varphi$ to a fixed-point free involution $\widehat{\varphi}$ on $M \times [-1,1]$ sending $(x,t) \mapsto (\varphi(x),-t)$.  We define $W$ to be the quotient of $M \times [-1,1]$ by the action of $\widehat{\varphi}$.  

We must show that $W$ is aspherical and $\partial W = M$.  Let $M'$ denote the universal cover of $M$.  Since $\widehat{\varphi}$ has no fixed points, the quotient map $q: M \times [-1,1] \rightarrow W$ is a covering map.  The universal covering space of $M \times [-1,1]$ is $M' \times [-1,1]$, and hence the composition $$M' \times [-1,1] \rightarrow M \times [-1,1] \stackrel{q}{\rightarrow} W,$$ is a covering map.  Since $M' \times [-1,1]$ is contractible (hence simply connected), it must be the universal covering space of $W$.  Thus $W$ is aspherical. 

Finally, the boundary of $M \times [-1,1]$ is $M \times \{-1\} \cup M \times \{1\}$, and hence the boundary of $W$ is canonically homeomorphic to $M$.

Up to homotopy the homomorphism $\pi_{1}(M)\to \pi_{1}(W)$ may be identified with the homomorphism $\pi_{1}(M)\to \pi_{1}(M/\varphi)$ induced by the 2-fold covering $M\to M/\varphi$, and hence is injective.
\end{proof}

\subsubsection{Cubical case}
This is the case Gromov \cite{Gromov1987} worked out.
For a given cubical complex $\mathcal{K}$, we will define a cell complex $\mathcal{M}(\mathcal{K})$ called the \textit{M\"obius band hyperbolization} of $\mathcal{K}$ by modifying the skeleta of $\K$, starting from the $2$-skeleton and inducting to the top-dimensional skeleton.  To do this, we will define the M\"obius band hyperbolizations of the cube $I^n$ and its boundary $\partial I^n$, and replace each face of $\K$ with its hyperbolization.   Having done this, we will be able to compute the Euler characteristic of the resulting complex and analyze it combinatorially.  We begin by following the presentation in \cite[Chapter 3.4]{Gromov1987}

We start with the $2$-cube $I^2$ and define its hyperbolization as $\mathcal{M}(I^2) = (\partial I^2 \times [-1,1])/\mathbb{Z}_2,$ where the $\mathbb{Z}_2$ action is defined by the involution $(t_1,t_2,t_3) \mapsto -(t_1,t_2,t_3)$.  Moreover, $\mathcal{M}(I^2)$ is the M\"obius band, and $\partial\mathcal{M}(I^2)$ is canonically isomorphic to $\partial I^2$.  Therefore, we hyperbolize the $2$-skeleton of $\K$ by replacing each $2$-face with a copy of $\mathcal{M}(I^2)$.  Since the vertices and edges of $\K$ are unchanged by this process, we define $\mathcal{M}(I^0) = I^0$ and $\mathcal{M}(I^1) = I^1$.

Inductively, suppose we have defined $\mathcal{M}(I^{n-1})$ and that we create $\mathcal{M}(\partial I^{n})$ by $\mathbb{Z}_2$-equivariantly inserting $\mathcal{M}(I^{n-1})$ for each of the $2n$ facets of $I^{n}$.
The antipodal involution on $\partial I^n$ then gives rise to a fixed-point free involution on $\mathcal{M}(\partial I^n)$. 
We define $\mathcal{M}(I^n) = (\mathcal{M}(\partial I^{n-1}) \times [-1,1])/(\mathbb{Z}_2)$ (and subsequently define $\mathcal{M}(\partial I^{n+1})$ by $\mathbb{Z}_2$-equivariantly inserting $\mathcal{M}(I^n)$ for each of the $2(n+1)$ facets of $I^{n+1}$). 

Making use of the contractibility of $I^{n}$ there is an inductively defined map $p_{n}: \mathcal{M}(I^{n})\to I^{n}$ such that the preimage of any face is precisely the hyperbolization of that face. In particular $\mathcal{M}(I^{n})$ has a face structure combinatorially equivalent to that of $I^{n}$.

If $\K$ is $n$-dimensional, then we define
$$ \mathcal{M}(\K) =  \bigcup_{\substack{F \in \K \\ \dim F = n}} \mathcal{M}(F)/\sim .$$ Here we identify the boundary facets of the $\mathcal{M}(F)$ in exactly the same pattern as the boundaries of the original facets $F$ were identified.

Alternatively one may view the result inductively as
$$ \mathcal{M}(\K) = \mathcal{M}(\K^{(n-1)}) \cup \bigcup_{\substack{F \in \K \\ \dim F = n}} \mathcal{M}(F),$$ where $\K^{(n-1)}$ denotes the $(n-1)$-skeleton of $\K$.

As a generalization of the M\"obius band hyperbolization, we can instead define $\mathcal{M}(I^2)$ to be any compact surface with a single boundary component with Euler characteristic $a \leq 0$, and otherwise continue as before.  We denote this hyperbolization by $\mathcal{M}_a(I^2)$, and, in general, $\mathcal{M}_{a}(\K)$. 

\begin{remark}
Although we do not need it here, one can systematically endow the hyperbolization $\mathcal{M}(\K)$ (as well as the other hyperbolizations discussed below) with a cubical structure. To do this properly one should cubically subdivide (as described in \ref{subsubsec:cubical}) the faces of $\K^{(n-1)}$, then extend over the the hyperbolized $n$-cells, etc.
\end{remark}

\subsubsection{Simplicial case}
We outline a similar treatment for the case of simplicial complexes.
For a given simplicial complex $\mathcal{K}$, we will  define a cell complex $\mathcal{M}(\mathcal{K})$ again called the \textit{M\"obius band hyperbolization} of $\mathcal{K}$ by modifying the skeleta of $\K$, starting from the $2$-skeleton and inducting to the top-dimensional skeleton.  To do this, we will inductively define the M\"obius band hyperbolizations of the $n$-simplex $\sigma^n$ and its boundary $\partial \sigma^n$, and replace each face of $\K$ with its hyperbolization.   The main difference between the simplicial case and the cubical case is that we must include a barycentric subdivision in the definition of the hyperbolized cells since the antipodal involution of the boundary of a simplex is not simplicial on the underlying simplex. This involution is defined on the vertices of the barycentric subdivision $(\sigma^{n})^{\prime}$ by $\hat{\mu}\mapsto \hat{\nu}$, where $\nu$ is the face of $\sigma^{n}$ complementary to the face $\mu$ and $\hat{\mu}$ denotes the barycenter of $\mu$.

We begin with the $2$-simplex $\sigma^2$ and define its hyperbolization as $\mathcal{M}(\sigma^2) = (\partial (\sigma{^2})^{\prime} \times [-1,1])/\mathbb{Z}_2,$ where the diagonal $\mathbb{Z}_2$ action is defined as above on the boundary of the simplex and by $t\mapsto -t$ on $[-1,1]$.  Again $\mathcal{M}(\sigma^2)$ is the M\"obius band, and $\partial\mathcal{M}(\sigma^2)$ is canonically isomorphic to $\partial (\sigma^2)^{\prime}$.  Therefore, we hyperbolize the $2$-skeleton of $\K$ by replacing each $2$-face with a copy of $\mathcal{M}(\sigma^2)$.  We simply define $\mathcal{M}(\sigma^0) = \sigma^0$ and $\mathcal{M}(\sigma^1) = (\sigma^1)'$.

Inductively, suppose we have defined $\mathcal{M}(\sigma^{n-1})$. We then create $\mathcal{M}(\sigma^{n})$ as follows. Replace each $(n-1)$-simplex of the barycentric subdivision $\partial(\sigma^{n})'$ with a copy of $\mathcal{M}(\sigma^{n-1})$, in a $\mathbb{Z}_{2}$ equivariant  way, so that the antipodal involution on $\partial (\sigma{^n})^{\prime}$ gives rise to a fixed-point free involution on $\mathcal{M}(\partial (\sigma^n)')$.  
We define $\mathcal{M}(\sigma^n) = (\mathcal{M}(\partial (\sigma{^{n-1}})') \times [-1,1])/(\mathbb{Z}_2)$.

Making use of the contractibility of $\sigma^{n}$ there is an inductively defined map $p_{n}: \mathcal{M}(\sigma^{n})\to \sigma^{n}$ such that the preimage of any face is precisely the hyperbolization of that face. In particular $\mathcal{M}(\sigma^{n})$ has a face structure combinatorially equivalent to that of $\sigma^{n}$.

If $\K$ is $n$-dimensional, then we define
$$ \mathcal{M}(\K) =  \bigcup_{\substack{F \in \K \\ \dim F = n}} \mathcal{M}(F)/\sim .$$ Here we identify the boundary facets of the $\mathcal{M}(F)$ in exactly the same pattern as the boundaries of the original facets $F$ were identified.

Alternatively we set
$$ \mathcal{M}(\K) = \mathcal{M}(\K^{(n-1)}) \cup \bigcup_{\substack{F \in \K \\ \dim F = n}} \mathcal{M}(F),$$ where $\K^{(n-1)}$ denotes the $(n-1)$-skeleton of $\K$.  

As before we can instead define $\mathcal{M}(\sigma^2)$ to be any compact surface with a single boundary component with Euler characteristic $a \leq 0$, and otherwise continue as above.  We denote this hyperbolization by $\mathcal{M}_a(\sigma^2)$, and, in general, $\mathcal{M}_{a}(\K)$.

\subsection{The Gromov hyperbolization}
The central idea behind  the present technique is the following observation of Gromov \cite{Gromov1987}. 

\begin{proposition}\label{prop:gromov}
Let $M$ be an aspherical $d$-manifold without boundary that admits an involution $r$ with fixed point set $B$ an aspherical $(d-1)$-manifold separating $M$ into two aspherical components, so that $M=A\cup_{B}r(A)$ with $A\cap r(A) =B$.  Assume $\pi_{1}(B)\to \pi_{1}(M)$ and $\pi_{1}(A)\to \pi_{1}(M)$ are injective. Then there is an aspherical $(d+1)$-manifold $W=\Delta(M)$ with $\partial W = M$ such that $\pi_{1}(M)\to \pi_{1}(W)$ is injective.  
\end{proposition}
\begin{proof}[Proof sketch]
We describe the manifold $W=\Delta(M)$ as $M\times [-1,1]/\sim$, where $(r(a),-1)\sim (r(a),1)$ for all $a\in A$. Alternatively $W$ is $M\times S^{1}$, cut open along $A\times \{x_{0}\}$. Then $\partial W=A\cup_{B} A\cong A\cup_{B} r(A) = M$. Up to homotopy we can think of $W$ as the union of $M\times [-1,1]$ and $r(A)\times [-1,1]$. Applying Whitehead's theorem we see that $W$ is aspherical. Finally, observe that $\pi_{1}(A\cup_{B} A)\to \pi_{1}(W)$ is injective by the HNN extension version of van Kampen's theorem. 
\end{proof}

\subsubsection{Simplicial case}
For the present hyperbolization Gromov described the main ideas in the simplicial case, where one takes advantage of barycentric subdivision. 

For a given simplicial complex $\mathcal{K}$, we  define a cell complex $\mathcal{G}(\mathcal{K})$ that we christen the \textit{Gromov hyperbolization} of $\mathcal{K}$ by modifying the skeleta of $\K$, starting from the $2$-skeleton and inducting to the top-dimensional skeleton.  To do this, we will inductively define the Gromov hyperbolizations of the $n$-simplex $\sigma^n$ and its boundary $\partial \sigma^n$, and replace each face of $\K$ with its hyperbolization.   We make use of a barycentric subdivision in order to have a simplicial symmetry that fixes a codimension-one subcomplex. As before this involution is defined on the vertices of the barycentric subdivision $(\sigma^{n})^{\prime}$ by $\hat{\mu}\mapsto \hat{\nu}$, where $\nu$ is the face of $\sigma^{n}$ complementary to the face $\mu$ and $\hat{\mu}$ denotes the barycenter of $\mu$.

We begin with the $2$-simplex $\sigma^2$ and define its hyperbolization to be  $\mathcal{G}(\sigma^2) = (\partial (\sigma{^2})^{\prime} \times [-1,1])/\mathbb{Z}_2,$ where the diagonal $\mathbb{Z}_2$ action is defined as above on the boundary of the simplex and by $t\mapsto -t$ on $[-1,1]$.  This time $\mathcal{G}(\sigma^2)$ is a once-punctured torus, and $\partial\mathcal{G}(\sigma^2)$ is canonically isomorphic to $\partial (\sigma^2)^{\prime}$.  Therefore, we hyperbolize the $2$-skeleton of $\K$ by replacing each $2$-face with a copy of $\mathcal{G}(\sigma^2)$.  We simply define $\mathcal{G}(\sigma^0) = \sigma^0$ and $\mathcal{G}(\sigma^1) = \sigma^1$. 

Inductively, suppose we have defined $\mathcal{G}(\sigma^{n-1})$, and we replace each $(n-1)$-simplex of $\partial(\sigma^{n})'$ with a copy of $\mathcal{G}(\sigma^{n-1})$, in such a way that the  involution $r$ on $\partial (\sigma{^n})^{\prime}$ induced by interchanging two vertices of $\sigma^{n}$ and leaving the remaining vertices of $\sigma^{n}$ fixed, gives rise to an involution on $\mathcal{G}(\partial (\sigma^n)')$ with a codimension-one fixed point set.  

We define $\mathcal{G}(\sigma^n) =\Delta (\mathcal{G}(\partial (\sigma{^{n-1}})') $, where $\Delta$ denotes the construction described in Proposition \ref{prop:gromov}.

Making use of the contractibility of $\sigma^{n}$ there is an inductively defined map $p_{n}: \mathcal{G}(\sigma^{n})\to \sigma^{n}$ such that the preimage of any face is precisely the hyperbolization of that face. In particular $\mathcal{G}(\sigma^{n})$ has a face structure combinatorially equivalent to that of $\sigma^{n}$.

If $\K$ is $n$-dimensional, then we define
$$ \mathcal{G}(\K) =  \bigcup_{\substack{F \in \K \\ \dim F = n}} \mathcal{G}(F)/\sim .$$ Here we identify the boundary facets of the $\mathcal{G}(F)$ in exactly the same pattern as the boundaries of the original facets $F$ were identified.

Alternatively we set
$$ \mathcal{G}(\K) = \mathcal{G}(\K^{(n-1)}) \cup \bigcup_{\substack{F \in \K \\ \dim F = n}} \mathcal{G}(F),$$ where $\K^{(n-1)}$ denotes the $(n-1)$-skeleton of $\K$.  Finally, we define the \textit{Gromov hyperbolization} of $\K$, denoted by $\mathcal{G}(\K)$, to be the complex obtained by hyperbolizing the cells of $\K$, starting with the $2$-skeleton and proceeding to its top-dimensional skeleton.

As before we can instead define $\mathcal{G}(\sigma^2)$ to be any compact surface with a single boundary component with Euler characteristic $a \leq 0$.  We denote this hyperbolization by $\mathcal{G}_a(\sigma^2)$, and, in general, $\mathcal{G}_{a}(\K)$.

\subsubsection{Cubical case}\label{subsubsec:cubical}
We adapt the preceding construction to the cubical case as well. The construction for cubical complexes is quite similar, where we understand the cubical barycentric subdivision of a cube $I^{n}$ to be given by
\[
(I^{n})'=(I')^{n}
\]
where $I'$ denotes the 1-complex with three vertices and two edges and the product is given the product cubical structure. Note, in particular, that $(I^{n})'$ admits a reflection $r$ that interchanges a pair of opposite facets and fixes a copy of $(I^{n-1})'$.

For a given cubical complex $\mathcal{K}$, we  define a cell complex $\mathcal{G}(\mathcal{K})$ that we again christen the \textit{Gromov hyperbolization} of $\mathcal{K}$ by modifying the skeleta of $\K$, starting from the $2$-skeleton and inducting to the top-dimensional skeleton.  To do this, we will inductively define the Gromov hyperbolizations of the $n$-cube $I^n$ and its boundary $\partial I^n$, and replace each face of $\K$ with its hyperbolization.   We make use of a cubical barycentric subdivision in order to have a cubical symmetry that fixes a codimension-one subcomplex as indicated above.

We begin with the $2$-cube $I^2$ and define its hyperbolization to be  $\mathcal{G}(I^2) = (\partial (I{^2})^{\prime} \times [-1,1])/\mathbb{Z}_2,$ where the diagonal $\mathbb{Z}_2$ action is defined as above on the boundary of the square and by $t\mapsto -t$ on $[-1,1]$.  Again $\mathcal{G}(I^2)$ is topologically a once-punctured torus, and $\partial\mathcal{G}(I^2)$ is canonically isomorphic to $\partial (I^2)^{\prime}$.  Therefore, we hyperbolize the $2$-skeleton of $\K$ by replacing each $2$-face with a copy of $\mathcal{G}(I^2)$.  We simply define $\mathcal{G}(I^0) = I^0$ and $\mathcal{G}(I^1) =I^1$. 

Inductively, suppose we have defined $\mathcal{G}(I^{n-1})$, and we replace each $(n-1)$-cube of $\partial(I^{n})'$ with a copy of $\mathcal{G}(I^{n-1})$, in such a way that the  involution $r$ on $\partial (I{^n})^{\prime}$ induced by interchanging two facets of $I^{n}$ as discussed above, gives rise to an involution $r$ on $\mathcal{G}(\partial (I^n)')$.

We define $\mathcal{G}(I^n) =\Delta (\mathcal{G}(\partial (I{^{n-1}})') $,  where $\Delta$ denotes the construction described in Proposition \ref{prop:gromov}.

Making use of the contractibility of $I^{n}$ there is an inductively defined map $p_{n}: \mathcal{G}(I^{n})\to I^{n}$ such that the preimage of any face is precisely the hyperbolization of that face. In particular $\mathcal{G}(I^{n})$ has a face structure combinatorially equivalent to that of $I^{n}$.

If $\K$ is $n$-dimensional, then we define
$$ \mathcal{G}(\K) =  \bigcup_{\substack{F \in \K \\ \dim F = n}} \mathcal{G}(F)/\sim .$$ Here we identify the boundary facets of the $\mathcal{G}(F)$ in exactly the same pattern as the boundaries of the original facets $F$ were identified.

Alternatively we set

$$ \mathcal{G}(\K) = \mathcal{G}(\K^{(n-1)}) \cup \bigcup_{\substack{F \in \K \\ \dim F = n}} \mathcal{G}(F),$$ where $\K^{(n-1)}$ denotes the $(n-1)$-skeleton of $\K$.  

As before we can instead define $\mathcal{G}(I^2)$ to be any compact surface with a single boundary component with Euler characteristic $a \leq 0$.  We denote this hyperbolization by $\mathcal{G}_a(I^2)$, and, in general, $\mathcal{G}_{a}(\K)$.

\section{The combinatorics of the M\"obius band hyperbolization} \label{sec:mobius}

\subsection{Euler characteristic of the cubical M\"obius band hyperbolization}

We begin by studying the effect of hyperbolizing the cells of a cubical complex on its Euler characteristic.  Suppose $\K$ is a cubical complex, and let $\mathcal{L}$ be the complex obtained by hyperbolizing the $i$-skeleton of $\K$.  Pick an $(i+1)$-face $F$ of $\K$, and let $\mathcal{L}'$ be the complex obtained from $\mathcal{L}$ by adding the hyperbolized cell $\mathcal{M}_a(F)$.  Then $\chi(\mathcal{L}') = \chi(\mathcal{L}) + \chi(\mathcal{M}_a(F)) - \chi(\mathcal{M}_a(\partial F))$.  By inductively repeating this process, we see that $$\chi(\mathcal{M}_a(\K)) = \sum_{k=0}^da_{\mathcal{M}}(k)f_k(\K),$$ where $a_{\mathcal{M}}(k) := \chi(\mathcal{M}_a(I^k)) - \chi(\mathcal{M}_a(\partial I^k))$.  We call the numbers $a_{\mathcal{M}}(k)$ for $k=0,1, 2, \ldots$ the \textit{hyperbolization coefficients} of the cubical M\"obius band hyperbolization.

Since $\chi(\mathcal{M}_a(I^n)) = \frac{1}{2} \chi(\mathcal{M}_a(\partial I^n))$, for any $n \geq 3$ we obtain the following relation:
\begin{eqnarray*}
a_{\mathcal{M}}(n) &=& \chi(\mathcal{M}_a(I^n)) - \chi(\mathcal{M}_a(\partial I^n)) \\
&=& \frac{1}{2} \chi(\mathcal{M}_a(\partial I^n)) - \chi(\mathcal{M}_a(\partial I^n)) \\
&=& -\frac{1}{2}\sum_{k=0}^{n-1}a_{\mathcal{M}}(k)f_k(\partial I^n).
\end{eqnarray*}
Thus, the hyperbolization coefficients are given by $a_{\mathcal{M}}(0) = 1, a_{\mathcal{M}}(1) = -1, a_{\mathcal{M}}(2) = a \leq 0$, and 
\begin{equation} \label{mobius-an-recursion}
a_{\mathcal{M}}(n)=-\frac{1}{2}\sum_{k=0}^{n-1}a_{\mathcal{M}}(k){n\choose k}2^{n-k},
\end{equation}
for all $n \geq 3$.  Values of $a_{\mathcal{M}}(n)$ for small values of $n$ are shown in the following table.  In particular, notice that $a_{\mathcal{M}}(2k) = -\frac{1}{2}\chi(\mathcal{M}_a(\partial I^{2k})) = 0$ for $k\ge 2$ by Poincar\'e duality. 

\renewcommand{\arraystretch}{1.5}
\begin{center}
\begin{table}[hbt]
\begin{tabular}{|c||c|c|c|c|c|c|c|} \hline
$n$ & $0$ & $1$ & $2$ & $3$ & $4$ & $5$ & $6$ \\ \hline
$a_{\mathcal{M}}(n)$ & $1$ & $-1$ & $a$ & $2-3a$ & 0 & $-16+20a$ & 0 \\ \hline
\end{tabular}
\caption{Values of $a_{\mathcal{M}}(n)$ for small $n$}
\end{table}
\end{center}

Let $\K$ be a $d$-dimensional cubical complex.  We write

\begin{eqnarray*}
\chi(\mathcal{M}_{a}(\K)) &=& \sum_{k=0}^da_{\mathcal{M}}(k)f_k(\K) \\
&=& \sum_{k=0}^d2^{-k}a_{\mathcal{M}}(k)\sum_{j=0}^k{d-j \choose d-k}h_j^{(sc)}(\K) \\
&=& \sum_{j=0}^d\left[ \sum_{k=j}^d2^{-k}{d-j \choose d-k}a_{\mathcal{M}}(k)\right]h_j^{(sc)}(\K),
\end{eqnarray*}
and define $$c_{\mathcal{M}}(j,d):= \sum_{k=j}^d2^{-k}{d-j \choose d-k}a_{\mathcal{M}}(k),$$ so that $$\chi(\mathcal{M}_{a}(\K)) = \sum_{j=0}^dc_{\mathcal{M}}(j,d)h_j^{(sc)}(\K).$$

Calculations of the values of the coefficients $c_{\mathcal{M}}(j,d)$ for small values of $d$ reveal some surprising skew-symmetry and monotonicity properties that we formalize in the next three propositions. Table~\ref{mobius-c-j-d-table} shows the values of $c_{\mathcal{M}}(j,d)$ for small values of $d$, which we will use as the basis for the inductive arguments that follow.

\renewcommand{\arraystretch}{1.5}
\begin{center}
\begin{table}[htb]
\begin{tabular}{|c||c|c|c|c|c|c|c|c|} \hline
$d\setminus j$ & $0$ & $1$ & $2$ & $3$ & $4$ & $5$ & $6$  \\ \hline
$2$ & $\frac{a}{4}$ & $\frac{a-2}{4}$ & $\frac{a}{4}$ & & & &  \\ \hline
$3$ & $\frac{3a-2}{8}$ & $\frac{a-2}{8}$ & $\frac{-a+2}{8}$ & $\frac{-3a+2}{8}$ & & &   \\ \hline
$4$ & $0$ & $\frac{-6a+4}{16}$ & $\frac{-8a+8}{16}$ & $\frac{-6a+4}{16}$ & $0$ & & \\ \hline
$5$ & $\frac{-20a+16}{32}$ & $\frac{-20a+16}{32}$ & $\frac{-8a+ 8}{32}$ & $\frac{8a-8}{32}$ & $\frac{20a-16}{32}$ & $\frac{20a-16}{32}$ &  \\ \hline
$6$ & $0$ & $\frac{40a-32}{64}$ & $\frac{80a-64}{64}$ & $\frac{96a-80}{64}$ & $\frac{80a-64}{64}$ & $\frac{40a-32}{64}$ & $0$ \\ \hline
\end{tabular}
\medskip
\caption{Values of $c_{\mathcal{M}}(j,d)$ for small $d$}
\label{mobius-c-j-d-table}
\end{table}
\end{center}

\begin{lemma}\label{mobius-c-recursion}
For all $d$ and all $0 \leq j < d$, $$c_{\mathcal{M}}(j,d) = c_{\mathcal{M}}(j+1,d) + c_{\mathcal{M}}(j,d-1).$$
\end{lemma}

\begin{proof}
By definition,
\begin{eqnarray*}
c_{\mathcal{M}}(j,d) &=& \sum_{k=j}^d2^{-k}{d-j \choose d-k}a_{\mathcal{M}}(k) \\
&=& \sum_{k=j}^d2^{-k}\left[{d-j-1 \choose d-k-1} + {d-j-1 \choose d-k}\right]a_{\mathcal{M}}(k) \\
&=& \sum_{k=j}^d2^{-k}{d-1-j \choose d-k}a_{\mathcal{M}}(k) + \sum_{k=j}^d2^{-k}{d-1-j \choose d-1-k}a_{\mathcal{M}}(k) \\
&=& \sum_{k=j+1}^d2^{-k}{d-j-1 \choose d-k}a_{\mathcal{M}}(k) + \sum_{k=j}^{d-1}2^{-k}{d-1-j \choose d-1-k}a_{\mathcal{M}}(k) \\
&=& c_{\mathcal{M}}(j+1,d) + c_{\mathcal{M}}(j,d-1).
\end{eqnarray*}
To get from the third to the fourth line in the above equation we use the fact that the $k=j$ term term of the first  summation vanishes since ${d-j-1 \choose d-j}=0$ and similarly the $k=d$ term vanishes from the second summation.
\end{proof}

\begin{lemma}\label{mobius-c-symmetry}
For all $j$ and all $d$, $$c_{\mathcal{M}}(j,d) = (-1)^dc_{\mathcal{M}}(d-j,d).$$
\end{lemma}

\begin{proof}
We prove the claim by induction on $d$ and then induction on $j$.  Table~\ref{mobius-c-j-d-table} shows that the claim holds for small values of $d$.  By induction (on $d$), we may suppose the claim holds for all values $c_{\mathcal{M}}(j,d-1)$.  Now we prove that the claim holds for all values $c_{\mathcal{M}}(j,d)$ by reverse induction on $j$.  When $j=d$, Equation \eqref{mobius-an-recursion} implies that
\begin{displaymath}
c_{\mathcal{M}}(0,d) = \sum_{k=0}^d2^{-k}{d \choose k}a_{\mathcal{M}}(k) = -\frac{1}{2^d}a_{\mathcal{M}}(d) = -c_{\mathcal{M}}(d,d).
\end{displaymath}
When $d$ is even, $a_{\mathcal{M}}(d) = 0$, so $c_{\mathcal{M}}(0,d) = -c_{\mathcal{M}}(d,d) = 0$; and when $d$ is odd, $c_{\mathcal{M}}(0,d) = -c_{\mathcal{M}}(d,d)$ as desired.  

Finally suppose $j < d$ and that $c_{\mathcal{M}}(j+1,d) = (-1)^dc_{\mathcal{M}}(d-j-1,d)$.  Then
\begin{eqnarray*}
c_{\mathcal{M}}(j,d) &=& c_{\mathcal{M}}(j+1,d) +  c_{\mathcal{M}}(j,d-1) \\
&=& (-1)^dc_{\mathcal{M}}(d-j-1,d) + (-1)^{d-1}c_{\mathcal{M}}(d-j-1,d-1) \\
&=& (-1)^d\left[c_{\mathcal{M}}(d-j,d) + c_{\mathcal{M}}(d-j-1,d-1)\right] \\
&& +(-1)^{d-1}c_{\mathcal{M}}(d-1-j,d-1) \\
&=& (-1)^dc_{\mathcal{M}}(d-j,d),
\end{eqnarray*}
where the second line follows from our inductive hypotheses and the third line from Lemma~\ref{mobius-c-recursion}.
\end{proof}

\begin{lemma}\label{mobius-c-monotony}
For all $m \geq 1$, we have 
\begin{eqnarray}
\label{mobius-0mod4} && 0 = c_{\mathcal{M}}(0,4m) \leq c_{\mathcal{M}}(1,4m) \leq \cdots \leq c_{\mathcal{M}}(2m,4m) \\
\label{mobius-1mod4} && 0 \leq c_{\mathcal{M}}(2m,4m+1) \leq c_{\mathcal{M}}(2m-1,4m+1) \leq \cdots \leq c_{\mathcal{M}}(0,4m+1) \\
\label{mobius-2mod4}&&  0 = c_{\mathcal{M}}(0,4m+2) \geq c_{\mathcal{M}}(1,4m+2) \geq \cdots \geq c_{\mathcal{M}}(2m+1,4m+2) \\
\label{mobius-3mod4} && 0 \geq c_{\mathcal{M}}(2m+1,4m+3) \geq c_{\mathcal{M}}(2m,4m+3) \geq \cdots \geq c_{\mathcal{M}}(0,4m+3).
\end{eqnarray}
\end{lemma}

\begin{proof}
We prove the claim by induction on $m$.  We will show that for each value of $m$,  Equation \eqref{mobius-0mod4} implies Equation \eqref{mobius-1mod4}, which implies Equation \eqref{mobius-2mod4}, which implies Equation \eqref{mobius-3mod4} which in turn implies \eqref{mobius-0mod4} for dimension $4(m+1)$.  Table \ref{mobius-c-j-d-table} shows that the claim holds for the values $c_{\mathcal{M}}(j,4)$.  

Suppose first that \eqref{mobius-0mod4} holds.  By Lemma \ref{mobius-c-symmetry},  $c_{\mathcal{M}}(2m,4m+1) = -c_{\mathcal{M}}(2m+1,4m+1)$ and hence
\begin{eqnarray*}
c_{\mathcal{M}}(2m,4m+1) &=& c_{\mathcal{M}}(2m+1,4m+1) + c_{\mathcal{M}}(2m,4m) \\
&=& -c_{\mathcal{M}}(2m,4m+1) + c_{\mathcal{M}}(2m,4m).
\end{eqnarray*}
Thus $c_{\mathcal{M}}(2m,4m+1) = \frac{1}{2}c_{\mathcal{M}}(2m,4m) \geq 0$ by \eqref{mobius-0mod4}.  Similarly, for any $j < 2m$, $$c_{\mathcal{M}}(j,4m+1) = c_{\mathcal{M}}(j+1,4m+1) + c_{\mathcal{M}}(j,4m) \geq c_{\mathcal{M}}(j+1,4m+1).$$  Thus Equation \eqref{mobius-1mod4} holds.

Next, we show that \eqref{mobius-2mod4} holds.  Since $$0 = c_{\mathcal{M}}(0,4m+2) = c_{\mathcal{M}}(1,4m+2) + c_{\mathcal{M}}(0,4m+1),$$ we see that $c_{\mathcal{M}}(1,4m+2) = -c_{\mathcal{M}}(0,4m+1) \leq 0$.  Again, for $1 \leq j \leq 2m$, $$c_{\mathcal{M}}(j+1,4m+2) = c_{\mathcal{M}}(j,4m+2) - c_{\mathcal{M}}(j,4m+1) \leq c_{\mathcal{M}}(j,4m+2),$$ and hence Equation \eqref{mobius-2mod4} holds.  

Showing that \eqref{mobius-2mod4} implies \eqref{mobius-3mod4} is the same argument that was used to show that \eqref{mobius-0mod4} implies \eqref{mobius-1mod4}; showing that \eqref{mobius-3mod4} implies \eqref{mobius-0mod4} for dimension $4m+4$ uses the same argument that was used to show that \eqref{mobius-1mod4} implies \eqref{mobius-2mod4}.  
\end{proof}

Putting all of this together, we are able to prove the main theorem in the case of the M\"obius band hyperbolizations of a cubical manifold.
\begin{theorem} \label{thm:cubical-mobius}
Let $\K$ be a closed cubical $2m$-manifold.  Then $$(-1)^m\chi(\mathcal{M}_{a}(\K)) \geq 0$$ for any integer $a\le 0$.
\end{theorem}
\begin{proof}
When $2m \equiv 0 \mod 4$, equation \eqref{mobius-0mod4} and Lemma~\ref{mobius-c-symmetry} imply that $c_{\mathcal{M}}(j,2m) \geq 0$ for all $j$.  When $2m \equiv 2 \mod 4$, Equation \eqref{mobius-2mod4} and Lemma~\ref{mobius-c-symmetry} imply that $c_{\mathcal{M}}(j,2m) \leq 0$ for all $j$.  Since $h_j^{(sc)}(\K) \geq 0$ for all $j$, it follows that
\begin{eqnarray*}
(-1)^m\chi(\mathcal{M}_{a}(\mathcal{K})) = \sum_{j=0}^{2m}c_{\mathcal{M}}(j,2m)h_j^{(sc)}(\mathcal{K}) \geq 0.
\end{eqnarray*}
\end{proof}

\subsection{Euler characteristic of the simplicial M\"obius band hyperbolization}

As in the case of the cubical M\"obius band hyperbolization, we begin by defining \textit{simplicial hyperbolization coefficients} $b_{\mathcal{M}}(n):=\chi(\mathcal{M}_a(\sigma^n)) - \chi(\mathcal{M}_a(\partial\sigma^n))$ so that 
\begin{eqnarray*}
b_{\mathcal{M}}(n) &=& \chi(\mathcal{M}_a(\sigma^n)) - \chi(\mathcal{M}_a(\partial\sigma^n)) \\
&=& -\frac{1}{2}\chi(\mathcal{M}_a(\partial\sigma^n)) \\
&=& -\frac{1}{2} \sum_{k=0}^{n-1}b_{\mathcal{M}}(k)f_k(\partial\sigma^n).
\end{eqnarray*}
Thus the simplicial hyperbolization coefficients are given by $b_{\mathcal{M}}(0) = 1, b_{\mathcal{M}}(1) = -1, b_{\mathcal{M}}(2) = a \leq 0$, and 
\begin{equation} \label{mobius-bn-recursion}
b_{\mathcal{M}}(n)=-\frac{1}{2}\sum_{k=0}^{n-1}b_{\mathcal{M}}(k){n+1\choose k+1},
\end{equation}
for all $n \geq 3$.  Values of $b_{\mathcal{M}}(n)$ for small values of $n$ are shown in the following table.  In particular, notice that $b_{\mathcal{M}}(2k) = -\frac{1}{2}\chi(\mathcal{M}_a(\partial \sigma^{2k})) = 0$ by Poincar\'e duality. 

\renewcommand{\arraystretch}{1.5}
\begin{center}
\begin{table} [htb]
\begin{tabular}{|c||c|c|c|c|c|c|c|} \hline
$n$ & $0$ & $1$ & $2$ & $3$ & $4$ & $5$ & $6$ \\ \hline
$b_{\mathcal{M}}(n)$ & $1$ & $-1$ & $a$ & $1-2a$ & 0 & $-3+5a$ & 0 \\ \hline
\end{tabular}
\caption{Values of $b_{\mathcal{M}}(n)$ for small $n$}
\end{table}
\end{center}

As in the cubical case, we can express the Euler characteristic of the M\"obius band hyperbolization of a simplicial $d$-manifold $\Delta$ as

\begin{eqnarray*}
\chi(\mathcal{M}_{a}(\Delta)) &=& \sum_{k=0}^d b_{\mathcal{M}}(k)f_k(\Delta) \\
&=& \sum_{k=0}^d b_{\mathcal{M}}(k)\frac{1}{k+1}\sum_{j=0}^k{d-j \choose d-k}\widetilde{h}_j(\Delta) \\
&=& \sum_{j=0}^d\left[ \sum_{k=j}^d\frac{1}{k+1}{d-j \choose d-k}b_{\mathcal{M}}(k)\right]\widetilde{h}_j(\Delta).
\end{eqnarray*}
We define $$s_{\mathcal{M}}(j,d):= \sum_{k=j}^d\frac{1}{k+1}{d-j \choose d-k}b_{\mathcal{M}}(k),$$ so that $$\chi(\mathcal{M}_{a}(\K)) = \sum_{j=0}^ds_{\mathcal{M}}(j,d)\widetilde{h}_j(\Delta).$$

The values of the coefficients $s_{\mathcal{M}}(j,d)$ reveal the same skew-symmetry and monotonicity properties that we saw in the cubical case. Table~\ref{mobius-s-j-d-table} shows the values of $s_{\mathcal{M}}(j,d)$ for small values of $d$.

\renewcommand{\arraystretch}{1.5}
\begin{center}
\begin{table}[htb]
\begin{tabular}{|c||c|c|c|c|c|c|c|c|} \hline
$d\setminus j$ & $0$ & $1$ & $2$ & $3$ & $4$ & $5$ & $6$  \\ \hline
$2$ & $\frac{a}{3}$ & $\frac{a}{3}-\frac{1}{2}$ & $\frac{a}{3}$ & & & &  \\ \hline
$3$ & $\frac{a}{2}-\frac{1}{4}$ & $\frac{a}{6}-\frac{1}{4}$ & $-\frac{a}{6}+\frac{1}{4}$ & $-\frac{a}{2}+\frac{1}{4}$ & & &   \\ \hline
$4$ & $0$ & $-\frac{a}{2}+\frac{1}{4}$ & $-\frac{2a}{3}+\frac{1}{2}$ & $-\frac{a}{2}+\frac{1}{4}$ & $0$ & & \\ \hline
$5$ & $-\frac{5a}{6}+\frac{1}{2}$ & $-\frac{5a}{6}+\frac{1}{2}$ & $-\frac{a}{3}+\frac{1}{4}$ & $\frac{a}{3}-\frac{1}{4}$ & $\frac{5a}{6}-\frac{1}{2}$ & $\frac{5a}{6}-\frac{1}{2}$ &  \\ \hline
$6$ & $0$ & $\frac{5a}{6}-\frac{1}{2}$ & $\frac{5a}{3}-1$ & $2a-\frac{5}{4}$ & $\frac{5a}{3}-1$ & $\frac{5a}{6}-\frac{1}{2}$ & $\quad 0\quad$ \\ \hline
\end{tabular}
\medskip
\caption{Values of $s_{\mathcal{M}}(j,d)$ for small $d$}
\label{mobius-s-j-d-table}
\end{table}
\end{center}

The proofs of the following simplicial analogues of Lemmas \ref{mobius-c-recursion}, \ref{mobius-c-symmetry} and \ref{mobius-c-monotony} are identical to their cubical counterparts. 

\begin{lemma} \label{mobius-s-recursion}
For all $d$ and all $0 \leq j < d$, $$s_{\mathcal{M}}(j,d) = s_{\mathcal{M}}(j+1,d) + s_{\mathcal{M}}(j,d-1).$$
\end{lemma}

\begin{lemma} \label{mobius-s-symmetry}
For all $j$ and all $d$, $$s_{\mathcal{M}}(j,d) = (-1)^ds_{\mathcal{M}}(d-j,d).$$
\end{lemma}

\begin{lemma}\label{mobius-s-monotony}
For all $m \geq 1$, we have 
\begin{eqnarray}
\label{s-mobius-0mod4} && 0 = s_{\mathcal{M}}(0,4m) \leq s_{\mathcal{M}}(1,4m) \leq \cdots \leq s_{\mathcal{M}}(2m,4m) \\
\label{s-mobius-1mod4} && 0 \leq s_{\mathcal{M}}(2m,4m+1) \leq s_{\mathcal{M}}(2m-1,4m+1) \leq \cdots \leq s_{\mathcal{M}}(0,4m+1) \\
\label{s-mobius-2mod4}&&  0 = s_{\mathcal{M}}(0,4m+2) \geq s_{\mathcal{M}}(1,4m+2) \geq \cdots \geq s_{\mathcal{M}}(2m+1,4m+2) \\
\label{s-mobius-3mod4} && 0 \geq s_{\mathcal{M}}(2m+1,4m+3) \geq s_{\mathcal{M}}(2m,4m+3) \geq \cdots \geq s_{\mathcal{M}}(0,4m+3).
\end{eqnarray}
\end{lemma}

Once again, these lemmas imply Hopf's formula holds for the M\"obius band hyperbolization of a simplicial manifold. 
\begin{theorem} \label{thm:simplicial-mobius}
Let $\mathcal{K}$ be a closed simplicial $2m$-manifold.  Then $$(-1)^m\chi(\mathcal{M}_{a}(\mathcal{K})) \geq 0$$ for any integer $a\le 0$.
\end{theorem}

\section{The combinatorics of the Gromov hyperbolization} \label{sec:gromov}

\subsection{Euler characteristic of the cubical Gromov hyperbolization}

As in the case of the M\"obius band hyperbolization, we define the cubical Gromov hyperbolization coefficient $a_{\mathcal{G}}(n):=\chi(\mathcal{G}_a(I^n)) - \chi(\mathcal{G}_a(\partial I^n))$ so that $\chi(\mathcal{G}_a(\K)) = \sum_{k=0}^da_{\mathcal{G}}(k)f_k(\K)$ for any cubical $d$-manifold $\K$.  We begin by deriving a recursive formula for these hyperbolization coefficients.  By our construction, $a_{\G}(0) = 1$, $a_{\G}(1) = -1$, and $a_{\G}(2) = a \leq -1$.  

We need to find a recursive formula for $a_{\mathcal{G}}(n)=\chi(\mathcal{G}_a(I^n)) - \chi(\mathcal{G}_a(\partial I^n))$. 

Let $X^{n}=\mathcal{G}_a(I^{n})$ and $Y^{n-1}=\mathcal{G}_{a}(\partial I^{n})$. Recall that by construction $Y^{n-1}$ admits a reflection $r$ with fixed point set $B^{n-1}$ isomorphic to $\mathcal{G}(I^{n-2})$, splitting $Y^{n-1}$ into two pieces $A^{n-1}$ and $r(A^{n-1})$, identified along $B^{n-1}$. Moreover $X^{n}$ is obtained from $Y^{n-1}\times [-1,1]$ by identifying $r(A^{n-1})\times\{-1\}$ with $r(A^{n-1})\times\{1\}$.

Then we have
\begin{align*}
a_{\mathcal{G}}(n)&=\chi(X^{n}) - \chi(Y^{n-1})\\
&=\chi(Y^{n-1}\times [-1,1])-\chi(A^{n-1})-\chi(Y^{n-1})\\
&= -\chi(A^{n-1}).
\end{align*}

On the other hand,
\begin{displaymath}
\chi(Y^{n-1})=2\chi(A^{n-1})-\chi(B^{n-2}),
\end{displaymath}
so that
\begin{align*}
\chi(A^{n-1})&=\frac{1}{2}\left(\chi(B^{n-2})+\chi(Y^{n-1})\right)\\
&=\frac{1}{2}\left(\chi(\mathcal{G}(\partial I^{n-1})+\chi(\mathcal{G}(\partial I^{n})\right).
\end{align*}

Thus by the inductive nature of the construction,
\begin{equation} \label{gromov-an-recursion}
a_{\G}(n) = -\frac{1}{2} \left(\sum_{k=0}^{n-2}a_{\G}(k)f_k(\partial I^{n-1}) + \sum_{k=0}^{n-1}a_{\G}(k)f_k(\partial I^n)\right).
\end{equation}

\renewcommand{\arraystretch}{1.5}

\begin{center}
\begin{table} [hbt]
\begin{tabular}{|c||c|c|c|c|c|c|c|c|} \hline
$n$ & $0$ & $1$ & $2$ & $3$ & $4$ & $5$ & $6$ & $7$ \\ \hline
$a_{\G}(n)$ & $1$ & $-1$ & $a$ & $2-3a$ & $2-3a$ & $-26+35a$ & $-26+35a$ & $594-791a$\\ \hline
\end{tabular}
\caption{Values of $a_{\G}(n)$ for small $n$}
\end{table}
\end{center}

As in the M\"obius band hyperbolization of a cubical complex, we define $$c_{\G}(j,d):=\sum_{k=j}^d2^{-k}{d-j \choose d-k}a_{\G}(k),$$ so that $\chi(\G_a(\K)) = \sum_{j=0}^d c_{\G}(j,d) h_j\scc(\K)$ for any cubical $d$-manifold $\K$.  Values of $c_{\mathcal{G}}(j,d)$ for small $d$ are shown in the following table: 

\renewcommand{\arraystretch}{1.5}
\begin{center}
\begin{table}[hbt]
\begin{tabular}{|c||c|c|c|c|c|c|c|} \hline 
$d \backslash j$ & $0$ & $1$ & $2$ & $3$ & $4$ & $5$ & $6$ \\ \hline
$2$ & $\frac{a}{4}$ & $\frac{a-2}{4}$ & $\frac{a}{4}$ & & & & \\ \hline
$3$ & $\frac{3a-2}{8}$ & $\frac{a-2}{8}$ & $\frac{-a+2}{8}$ & $\frac{-3a+2}{8}$ & & & \\ \hline
$4$ & $\frac{-3a+2}{16}$ & $\frac{-9a+6}{16}$ & $\frac{-11a+10}{16}$ & $\frac{-9a+6}{16}$ & $\frac{-3a+2}{16}$ & & \\ \hline
$5$ & $\frac{-35a+26}{32}$ & $\frac{-29a+22}{32}$ & $\frac{-11a+10}{32}$ & $\frac{11a-10}{32}$ & $\frac{29a-22}{32}$ & $\frac{35a-26}{32}$ & \\ \hline
$6$ & $\frac{35a-26}{64}$ & $\frac{105a-78}{64}$ & $\frac{163a-126}{64}$ & $\frac{185a-142}{64}$ & $\frac{163a-126}{64}$ & $\frac{105a-78}{64}$ & $\frac{35a-26}{64}$ \\ \hline
\end{tabular}
\label{gromov-c-j-d}
\caption{Values of $c_{\G}(j,d)$ for small $j$ and $d$.}
\end{table}
\end{center}

Again we will establish Gromov versions of Lemmas \ref{mobius-c-recursion}, \ref{mobius-c-symmetry}, and \ref{mobius-c-monotony}. The proof of the following recursive formula is identical to that of Lemma \ref{mobius-c-recursion}.

\begin{lemma} \label{gromov-c-recursion}
For all $d$ and all $0 \leq j < d$, $$c_{\G}(j,d) = c_{\G}(j+1,d) + c_{\G}(j,d-1).$$
\end{lemma}

In order to show that the coefficients $c_{\G}(j,d)$ are skew-symmetric and monotone as we have before, we require an additional lemma in this case. 

\begin{lemma} \label{gromov-c-top-terms}
For all $d$, 
$$2[c_{\G}(0,d) + c_{\G}(d,d)] = c_{\G}(d-1,d-1) - c_{\G}(0,d-1).$$
\end{lemma}
\begin{proof}
We use the defining equations for the coefficients $c_{\G}(j,d)$ to explicitly compute that 
\begin{equation} \label{gromov-c-0-d}
c_{\G}(0,d) = \frac{1}{2^d} \sum_{k=0}^d {d \choose k} 2^{d-k}a_{\G}(k),
\end{equation} 
and
\begin{equation} \label{gromov-c-d-d}
c_{\G}(d,d) = \frac{1}{2^d} a_{\G}(d).
\end{equation}
Adding $a_{\G}(d) + a_{\G}(d-1)$ to both sides of the recursion equation \eqref{gromov-an-recursion} for the hypberbolization coefficients  gives 
\begin{displaymath}
-a_{\G}(d) + a_{\G}(d-1) = \sum_{k=0}^{d-1} {d-1 \choose k}2^{d-1-k} a_{\G}(k) + \sum_{k=0}^d{d \choose k}2^{d-k}a_{\G}(k).
\end{displaymath}
By Equation \eqref{gromov-c-d-d}, the left side of this equation is $-2^dc_{\G}(d,d) + 2^{d-1} c_{\G}(d-1,d-1)$ and by Equation \eqref{gromov-c-0-d}, the right side is $2^{d-1}c_{\G}(0,d-1) + 2^dc_{\G}(0,d)$.
\end{proof}

\begin{lemma} \label{gromov-c-symmetry}
For all $j$ and all $d$, we have $$c_{\G}(j,d) = (-1)^dc_{\G}(d-j,d).$$
\end{lemma}
\begin{proof}
We begin by proving the lemma in the case that $j=d$.  We prove this claim by induction, with the base cases handled by the data in Table \ref{gromov-c-j-d}.  By Lemma \ref{gromov-c-recursion}, 
\begin{eqnarray*}
c_{\G}(0,d) &=& c_{\G}(1,d) + c_{\G}(0,d-1) \\
&=& c_{\G}(2,d) + c_{\G}(1,d-1) + c_{\G}(0,d-1) \\
&=& \cdots \\
&=& c_{\G}(d,d) + \sum_{j=0}^{d-1} c_{\G}(j,d-1).
\end{eqnarray*}
When $d$ is even, we know $\chi(\G(\partial I^d)) = 0$ by Poincar\'e duality.  On the other hand, $h_j\scc(\partial I^d) = 2^{d}$ for all $j$ and so $$0 = \chi(\G(\partial I^d)) = \sum_{j=0}^{d-1}c_{\G}(j,d) 2^{d}.$$  Thus $c_{\G}(0,d) = c_{\G}(d,d)$ when $d$ is even. 

When $d$ is odd, $c_{\G}(0,d-1) = c_{\G}(d-1,d-1)$ by our inductive hypothesis.  Thus Lemma \ref{gromov-c-top-terms} implies $c_{\G}(0,d) = -c_{\G}(d,d).$  

Finally, to prove the lemma for $j < d$, we proceed by induction on $d$ and reverse induction on $j$ as in the proof of Lemma \ref{mobius-c-symmetry}: 
\begin{eqnarray*}
c_{\G}(j,d) &=& c_{\G}(j+1,d) + c_{\G}(j,d-1) \\
&=& (-1)^dc_{\G}(d-j-1,d) + (-1)^{d-1}c_{\G}(d-j-1,d-1) \\
&=& (-1)^d[c_{\G}(d-j,d) + c_{\G}(d-j-1,d-1)] + (-1)^{d-1}c_{\G}(d-j-1,d-1) \\
&=& (-1)^dc_{\G}(d-j,d).
\end{eqnarray*}
\end{proof}

Finally, we prove the Gromov version of Lemma \ref{mobius-c-monotony}.  We must account for the the even-dimensional cases where $c_{\G}(0,4m)$ and $c_{\G}(0,4m+2)$ are no longer guaranteed to be zero.

\begin{lemma}\label{gromov-c-monotony}
For all $m \geq 1$, and $a \leq -1$ we have 
\begin{eqnarray}
\label{gromov-0mod4} && 0 \leq c_{\G}(0,4m) \leq c_{\G}(1,4m) \leq \cdots \leq c_{\G}(2m,4m) \\
\label{gromov-1mod4} && 0 \leq c_{\G}(2m,4m+1) \leq c_{\G}(2m-1,4m+1) \leq \cdots \leq c_{\G}(0,4m+1) \\
\label{gromov-2mod4} &&  0 \geq  c_{\G}(0,4m+2) \geq c_{\G}(1,4m+2) \geq \cdots \geq c_{\G}(2m+1,4m+2) \\
\label{gromov-3mod4} && 0 \geq c_{\G}(2m+1,4m+3) \geq c_{\G}(2m,4m+3) \geq \cdots \geq c_{\G}(0,4m+3).
\end{eqnarray}
\end{lemma}

\begin{proof}
As before prove the claim by induction on $m$, showing that \eqref{gromov-0mod4} implies \eqref{gromov-1mod4}  which in turn implies \eqref{gromov-2mod4} which implies \eqref{gromov-3mod4} which implies \eqref{gromov-0mod4} in dimension $4(m+1)$.  Table \ref{gromov-c-j-d} shows that the claim holds for the values $c_{\mathcal{M}}(j,4)$, and hence \eqref{gromov-0mod4} is valid when $m=1$.  In many cases, the computations are identical to those done in the proof of Lemma \ref{mobius-c-monotony}, so we will omit those calculations. 

The proof that  the inequalities \eqref{gromov-0mod4} imply the inequalities \eqref{gromov-1mod4} relies only on the recursion in Lemma \ref{gromov-c-recursion} and skew-symmetry of Lemma \ref{gromov-c-symmetry} and is identical to the one in Lemma \ref{mobius-c-monotony}.  

In order to show that  the inequalities \eqref{gromov-1mod4} imply  the inequalities\eqref{gromov-2mod4}, we must first show that $c_{\G}(0,4m+2) \leq 0$.  This follows by applying Lemma \ref{gromov-c-symmetry} to Lemma \ref{gromov-c-top-terms} and seeing that $2c_{\G}(0,4m+2) = - c_{\G}(0,4m+1)$.  The remainder of the proof continues as before using the recursion \eqref{gromov-c-recursion}. 

The proof that the inequalities \eqref{gromov-2mod4} imply \eqref{gromov-3mod4} is identical to the M\"obius case.  Once again, in order to show that \eqref{gromov-3mod4} implies \eqref{gromov-0mod4} in dimension $4m+4$ we must first establish that $c_{\G}(0,4m+4) \geq 0$, but again this follows by applying Lemma \ref{gromov-c-symmetry} to Lemma \ref{gromov-c-top-terms}.
\end{proof}

As in the cubical case, Lemmas \ref{gromov-c-symmetry} and \ref{gromov-c-monotony}, together with the fact that the short $h$-numbers of a cubical complex are nonnegative, imply the Sign Conjecture for the Gromov hyperbolization. 

\begin{theorem} \label{thm:cubical-gromov}
Let $\K$ be a closed cubical $2m$-manifold.  Then $$(-1)^m\chi(\G_a(\K)) \geq 0.$$
\end{theorem}

\subsection{Euler characteristic of the simplicial Gromov hyperbolization}

As before we define the simplicial Gromov hyperbolization coefficients $b_{\G}(n):= \chi(\G_a(\sigma^n)) - \chi(\G_a(\partial\sigma^n))$ so that $\chi(\G_a(\Delta)) = \sum_{k=0}^db_{\G}(k)f_k(\Delta)$ for any simplicial $d$-manifold $\Delta$.  

We need to find a recursive formula for $b_{\mathcal{G}}(n)=\chi(\mathcal{G}_a(\sigma^n)) - \chi(\mathcal{G}_a(\partial \sigma^n))$. 

As in the cubical case, let $X^{n}=\mathcal{G}_a(\sigma^{n})$ and $Y^{n-1}=\mathcal{G}_{a}(\partial \sigma^{n})$. Recall that by construction $Y^{n-1}$ admits a reflection $r$ with fixed point set $B^{n-1}$ isomorphic to $\mathcal{G}(\sigma^{n-2})$, splitting $Y^{n-1}$ into two pieces $A^{n-1}$ and $r(A^{n-1})$, identified along $B^{n-1}$. Moreover $X^{n}$ is obtained from $Y^{n-1}\times [-1,1]$ by identifying $r(A^{n-1})\times\{-1\}$ with $r(A^{n-1})\times\{1\}$.

By repeating the analysis that was used in the cubical Gromov hyperbolization replacing the $n$-cube with the $n$-simplex throughout, we once again derive the recursion
\begin{equation} \label{gromov-bn-recursion}
b_{\G}(n) = -\frac{1}{2} \left( \sum_{k=0}^{n-2} b_{\G}(k)f_k(\partial\sigma^{n-1}) + \sum_{k=0}^{n-1} b_{\G}(k) f_k(\partial\sigma^n)\right).
\end{equation} 

\begin{center}
\begin{table} [hbt]
\begin{tabular}{|c||c|c|c|c|c|c|c|c|} \hline
$n$ & $0$ & $1$ & $2$ & $3$ & $4$ & $5$ & $6$ & $7$\\ \hline 
$b_{\G}(n)$ & $1$ & $-1$ & $a$ & $1-2a$ & $1-2a$ & $-6+11a$ & $-6+11a$ & $55-100a$\\ \hline
\end{tabular}
\caption{Values of $b_{\G}(n)$ for small $n$}
\end{table}
\end{center}

As in the M\"obius band hyperbolization of a simplicial complex, we define $$s_{\G}(j,d) = \sum_{k=j}^d \frac{1}{k+1}{d-j \choose d-k}b_{\G}(k),$$ so that $\chi(\G_a(\Delta)) = \sum_{j=0}^ds_{\G}(j,d)\widetilde{h}_j(\Delta)$ for any simplicial $d$-manifold $\Delta$.  

\renewcommand{\arraystretch}{1.5}
\begin{center}
\begin{table} [hbt] \label{gromov-s-j-d}
\begin{tabular}{|c||c|c|c|c|c|c|c|} \hline
$d \backslash j$ & $0$ & $1$ & $2$ & $3$ & $4$ & $5$ & $6$ \\ \hline 
$2$ & $\frac{a}{3}$ & $\frac{a}{3}-\frac{1}{2}$ & $\frac{a}{3}$ & & & & \\ \hline
$3$ & $\frac{a}{2}-\frac{1}{4}$ & $\frac{a}{6}-\frac{1}{4}$ & $-\frac{a}{6}+\frac{1}{4}$ & $-\frac{a}{2}+\frac{1}{4}$ & & & \\ \hline
$4$ & $-\frac{2a}{5} + \frac{1}{5}$ & $-\frac{9a}{10} + \frac{9}{20}$ & $-\frac{16a}{15}+\frac{7}{10}$ & $-\frac{9a}{10} + \frac{9}{20}$ & $-\frac{2a}{5} + \frac{1}{5}$ & & \\ \hline
$5$ & $-\frac{11a}{6}+1$ & $-\frac{43a}{30}+\frac{4}{5}$ & $-\frac{8a}{15}+\frac{7}{20}$ & $\frac{8a}{15}-\frac{7}{20}$ & $\frac{43a}{30}-\frac{4}{5}$ & $\frac{11a}{6}-1$ & \\ \hline
$6$ & $\frac{11a}{7}-\frac{6}{7}$ & $\frac{143a}{42}-\frac{13}{7}$ & $\frac{508a}{105}-\frac{93}{35}$ & $\frac{188a}{35}-\frac{421}{140}$ & $\frac{508a}{105}-\frac{93}{35}$ & $\frac{143a}{42}-\frac{13}{7}$ & $\frac{11a}{7}-\frac{6}{7}$ \\ \hline
\end{tabular}
\caption{Values of $s_{\G}(j,d)$ for small $j$ and $d$}
\end{table}
\end{center}

The proofs of the following lemmas are identical to their analogues for the cubical Gromov hyperbolization, and we omit their proofs. 

\begin{lemma} \label{gromov-s-recursion}
For all $d$ and all $0 \leq j <d$, $$s_{\mathcal{G}}(j,d) = s_{\mathcal{G}}(j+1,d) + s_{\mathcal{G}}(j,d-1).$$
\end{lemma}

\begin{lemma} \label{gromov-s-top-terms}
For all $d$, $$(d+1)\left(s_{\mathcal{G}}(0,d) + s_{\mathcal{G}}(d,d)\right) = d\left(s_{\mathcal{G}}(d-1,d-1)-s_{\mathcal{G}}(0,d-1)\right).$$
\end{lemma}

\begin{lemma} \label{gromov-s-symmetry}
For all $j$ and all $d$ we have $$s_{\mathcal{G}}(j,d) = (-1)^ds_{\mathcal{G}}(d-j,d).$$
\end{lemma}

\begin{lemma} \label{gromov-s-monotony}
For all $m \geq 1$, we have 
\begin{eqnarray*}
\label{simp-gromov-0mod4} && 0 \leq s_{\mathcal{G}}(0,4m) \leq s_{\mathcal{G}}(1,4m) \leq \cdots \leq s_{\mathcal{G}}(2m,4m) \\ 
\label{simp-gromov-1mod4} && 0 \leq s_{\mathcal{G}}(2m,4m+1) \leq s_{\mathcal{G}}(2m-1,4m+1) \leq \cdots \leq s_{\mathcal{G}}(0,4m+1) \\ 
\label{simp-gromov-2mod4} &&  0 \geq s_{\mathcal{G}}(0,4m+2) \geq s_{\mathcal{G}}(1,4m+2) \geq \cdots \geq s_{\mathcal{G}}(2m+1,4m+2) \\ 
\label{simp-gromov-3mod4} && 0 \geq s_{\mathcal{G}}(2m+1,4m+3) \geq s_{\mathcal{G}}(2m,4m+3) \geq \cdots \geq s_{\mathcal{G}}(0,4m+3). 
\end{eqnarray*}
\end{lemma}

As before, Lemmas \ref{gromov-s-symmetry} and \ref{gromov-s-monotony}, together with the fact that the short simplicial $h$-numbers of a closed manifold are nonnegative, prove the Sign Conjecture for Gromov hyperbolizations.

\begin{theorem} \label{thm:simplicial-gromov}
Let $\mathcal{K}$ be a closed simplicial $2m$-manifold.  Then $$(-1)^m\chi(\G_a(\mathcal{K})) \geq 0.$$
\end{theorem}

\section{Combinatorial properties of the hyperbolization coefficients} \label{sec:gen-funs}

For certain hyperbolization functors, the hyperbolization coefficients have interesting combinatorial interpretations when we specialize to the case that $a=0$ or $a=1$.

\subsection{The cubical M\"obius band hyperbolization}

When $a=0$, the coefficients $a_{\mathcal{M}}(n)$ are $$1,-1,0,2,0,-16,0,272,0,\ldots.$$  The sequence $1,2,16,272,\ldots,$  of ``tangent numbers" \cite[A000182]{OEIS} counts the number of permutations $\pi$ on $\{1,2,\ldots,2n-1\}$ such that $\pi_1<\pi_2>\pi_3<\cdots$.  The exponential generating function for this sequence is the Taylor series expansion of $\tan(x)$.  

Next, consider the Taylor series expansion $$\tan(x) = \sum_{n \geq 1}t_n\frac{x^{2n-1}}{(2n-1)!}.$$

Since $1-\tanh(x) = 1 + \mathbf{i}\tan(\mathbf{i} x)$\footnote{We are grateful to Lara Pudwell for pointing out this identity, which simplified the resulting computation.},
\begin{eqnarray*}
1-\tanh(x) &=& 1 + \mathbf{i}\tan(\mathbf{i}x) \\
&=& 1 + \mathbf{i}\sum_{n \geq 1}t_n\mathbf{i}^{2n-1}\frac{x^{2n-1}}{(2n-1)!} \\
&=& 1 + \sum_{n \geq 1}t_n\mathbf{i}^{2n}\frac{x^{2n-1}}{(2n-1)!} \\
&=& 1 + \sum_{n \geq 1}(-1)^nt_n\frac{x^{2n-1}}{(2n-1)!}.
\end{eqnarray*}

Let $F(x) = 1-\tanh(x)$.  We will show that $a_{\mathcal{M}}(n) = F^{(n)}(0)$, and hence $a_{\mathcal{M}}(2n-1) = (-1)^nt_n$ for all $n \geq 1$.  Thus the hyperbolic coefficients of odd index, $a_{\mathcal{M}}(2n-1)$, are signed versions of these well-studied combinatorial statistics. 

\begin{theorem}
Let $F(x) = \frac{2}{1+e^{2x}} = 1-\tanh(x)$. Then $a_{\mathcal{M}}(n) = F^{(n)}(0).$
\end{theorem}
\begin{proof}
By rewriting $F(x)$ as $(1+e^{2x})F(x) = 2$, a simple inductive argument shows that $$-F^{(n)}(x) = \sum_{k=0}^n2^{n-k}{n \choose k}e^{2x}F^{(k)}(x),$$

so 
\begin{equation}\label{gf-recursion}
-F^{(n)}(0) = \sum_{k=0}^n2^{n-k}{n \choose k}F^{(k)}(x).
\end{equation}

Since $F(0) = 1$, Equation \eqref{gf-recursion} shows that $F^{(n)}(0) = a_{\mathcal{M}}(n)$.  

\end{proof}

\subsection{The simplicial M\"obius band hyperbolization}

When $a=0$, the coefficients $b_{\mathcal{M}}(n)$ are $$1,-1,0,1,0,-3,0,17,0,-155,\ldots.$$  This is the sequence $1,1,3,17,155,\ldots$ of Genocchi numbers \cite[A110501]{OEIS}, which appear as the unsigned nonzero coefficients of the Taylor series expansion of the function $$G(x) = -x\tanh(\frac{x}{2}) = -x\frac{e^x-1}{e^x+1} = -\frac{x^2}{2!} + \frac{x^4}{4!} - \frac{3x^6}{6!} + \frac{17x^8}{8!} + \ldots.$$

\begin{theorem}
Let $G(x) = -x\frac{e^x-1}{e^x+1}$.  Then $b_{\mathcal{M}}(2n-1) = G^{(2n)}(0)$ for all $n \geq 1$.  
\end{theorem}

\begin{proof}
We can rewrite the defining equation for $G(x)$ as $(e^x+1)G(x) = -x(e^x-1)$.  Implicitly differentiating this equation shows that for any $m \geq 2$, 
\begin{equation} \label{gromov-G-implicit-diff}
(e^x+1)G^{(m)}(x) + \sum_{k=0}^{m-1}{m \choose k}G^{(k)}(x)e^x = -(x+m)e^x.
\end{equation}
Further, $G(0) = 0$ and $G(x)$ is an odd function, meaning $G^{(2k+1)}(0) = 0$ for all $k$.  Thus evaluating Equation \eqref{gromov-G-implicit-diff} with $m = 2n$ and $x=0$ gives
\begin{eqnarray*}
-2n &=& 2G^{(2n)}(0) + \sum_{k=0}^{2n-1}{2n \choose k}G^{(k)}(0) \\
&=& 2G^{(2n)}(0) + \sum_{k=1}^{n-1}{2n \choose 2k}G^{(2k)}(0).
\end{eqnarray*}
Similarly, since $b_{\mathcal{M}}(0) = 1$ and $b_{\mathcal{M}}(2k) =0$ for all $k$, Equation \eqref{mobius-bn-recursion} gives
\begin{eqnarray*}
-2b_{\mathcal{M}}(2n-1) &=& \sum_{k=0}^{2n-2}{2n \choose k+1}b_{\mathcal{M}}(k) \\
&=& 2n + \sum_{k=1}^{n-1}{2n \choose 2k}b_{\mathcal{M}}(2k-1).
\end{eqnarray*}
Thus $b_{\mathcal{M}}(2n-1) = G^{(2n)}(0)$ for all $n \geq 1$.
\end{proof}

\subsection{The cubical Gromov hyperbolization}

When $a=-1$, the coefficients $a_{\mathcal{G}}(n)$ are $$1, -1,-1, 5, 5, -61, -61, 1385, 1385, -50521, -50521,\dots.$$
The corresponding sequence $1,1,5,61,1385, 50521,\dots $ \cite[A000364]{OEIS} is the sequence of Euler numbers or ``secant numbers," which count the number of permutations $\pi$ on $\{1,2,\ldots, 2n\}$ such that $\pi_1 < \pi_2 > \pi_3 < \cdots$.  The exponential generating for this sequence is the Taylor series expansion of the function $y = \sec(x)$.  A signed version of these numbers also occurs in \cite[A28296]{OEIS}.  As in the case of the cubical M\"obius band hyperbolization, we use the identity $\sec(\mathbf{i}x) = \sech(x) =  \frac{2e^x}{e^{2x}+1}$ whose Taylor series expansion is $$\frac{2e^x}{e^{2x}+1} = 1 - \frac{x^2}{2!} + \frac{5x^4}{4!} - \frac{61x^6}{6!} + \cdots.$$

\begin{theorem} \label{gromov-a-are-euler}
Let $$\sech(x) = \sum_{n \geq 0} s_n \frac{x^n}{n!}.$$ Then $s_{2n+1} = 0$ and $s_{2n} = a_{\mathcal{G}}(2n)$ for all $n \geq 0$. 
\end{theorem}

In order to prove this theorem, we will manipulate the recursion formula defining the sequence $a_{\mathcal{G}}(n)$ and show that it coincides with a recursion defining the sequence $s_n$. First we require three lemmas.

\begin{lemma} \label{gromov-a-even=odd}
For all $n \geq 1$, $$a_{\mathcal{G}}(2n) = a_{\mathcal{G}}(2n-1).$$
\end{lemma}

\begin{proof}
By Lemma \ref{gromov-c-top-terms}, 
\begin{equation} \label{even=odd-eqn}
2[c_{\mathcal{G}}(0,2n) + c_{\mathcal{G}}(2n,2n)] = c_{\mathcal{G}}(2n-1,2n-1) - c_{\mathcal{G}}(0,2n-1).
\end{equation}
Moreover, by Proposition \ref{gromov-c-symmetry}, $c_{\mathcal{G}}(0,2n) = c_{\mathcal{G}}(2n,2n)$ and $c_{\mathcal{G}}(2n-1,2n-1) = -c_{\mathcal{G}}(0,2n-1)$.  Finally, $c_{\mathcal{G}}(d,d) = \frac{1}{2^d}a_{\mathcal{G}}(d)$ for any $d$ by the equation defining the coefficients $c_{\mathcal{G}}(j,d)$.  Thus the left side of Equation \eqref{even=odd-eqn} simplifies to $2\cdot \frac{1}{2^{2n}}a_{\mathcal{G}}(2n),$ and the right side simplifies to $\frac{1}{2^{2n-1}}a_{\mathcal{G}}(2n-1)$.  
\end{proof}

\begin{lemma} \label{gromov-a-simplification}
For all $n \geq 1$, $$\sum_{k=0}^{2n-1}a_{\mathcal{G}}(k){2n \choose k}2^{2n-k} = 0.$$
\end{lemma}
\begin{proof}
We prove the claim by induction on $n$.  A simple calculation verifies the result holds when $n=1$, so suppose $n > 1$ and the result holds inductively.  We use the recursion \eqref{gromov-an-recursion} to write
\begin{eqnarray*}
-2a_{\mathcal{G}}(2n) &=& \sum_{k=0}^{2n-2} a_{\mathcal{G}}(k)2^{2n-k-1}{2n-1 \choose k} + \sum_{k=0}^{2n-1}a_{\mathcal{G}}(k)2^{2n-k}{2n \choose k} \text{ and } \\
-2a_{\mathcal{G}}(2n-1) &=& \sum_{k=0}^{2n-3}a_{\mathcal{G}}(k)2^{2n-k-2}{2n-2 \choose k} + \sum_{k=0}^{2n-2}a_{\mathcal{G}}(k)2^{2n-k-1}{2n-1 \choose k}.
\end{eqnarray*}
These two quantities are equal by Lemma \ref{gromov-a-even=odd}.  The first summation in the top line is equal to the second summation in the bottom line.  The first summation in the bottom line is equal to $0$ by our inductive hypothesis.  Thus the second summation in the top line must be equal to $0$ as well.
\end{proof}

\begin{lemma} 
For all $n \geq 1$, 
\begin{equation} \label{gromov-a-good-recursion}
-2a_{\mathcal{G}}(2n) = \sum_{k=0}^{n-1}a_{\mathcal{G}}(2k)2^{2n-2k-1}\left[2{2n \choose 2k}-{2n-1\choose 2k}\right].
\end{equation}
\end{lemma}
\begin{proof}
We begin with the recursion \eqref{gromov-an-recursion}
\begin{eqnarray*}
-2a_{\mathcal{G}}(2n) &=& \sum_{k=0}^{2n-2}a_{\mathcal{G}}(k)2^{2n-k-1}{2n-1 \choose k} + \sum_{k=0}^{2n-1}a_{\mathcal{G}}(k)2^{2n-k}{2n \choose k} \\
&=& \sum_{k=0}^{2n-2}a_{\mathcal{G}}(k)2^{2n-k-1}{2n-1 \choose k} \text{ by Lemma \ref{gromov-a-simplification} } \\
&=& 2^{2n-1}a_{\mathcal{G}}(0)  \\
&& + \sum_{k=1}^{n-1}\left[2^{2n-2k}{2n-1 \choose 2k-1}a_{\mathcal{G}}(2k-1) + 2^{2n-2k-1}{2n-1 \choose 2k}a_{\mathcal{G}}(2k)\right]\\
&=& 2^{2n-1}a_{\mathcal{G}}(0) + \sum_{k=1}^{n-1}2^{2n-2k-1}a_{\mathcal{G}}(2k)\left[2{2n-1 \choose 2k-1} + {2n-1 \choose 2k}\right],
\end{eqnarray*}
where the last line comes from Lemma \ref{gromov-a-even=odd}.  The result follows since $$2{2n-1 \choose 2k-1} + {2n-1 \choose 2k} = 2{2n \choose 2k} - {2n-1 \choose 2k},$$ for any $k \geq 0$.
\end{proof}

Now we are ready to prove Theorem \ref{gromov-a-are-euler}.
\begin{proof}
We consider the function $\widetilde{F}(x) = \frac{2e^x}{e^{2x}+1}$ so that $s_k = \widetilde{F}^{(k)}(0)$ for all $k$.  We will show that the even coefficients in the Taylor series expansion of $y$ satisfy the recursion given in Equation \eqref{gromov-a-good-recursion}.  We begin by rewriting $2e^x = (e^{2x}+1)\widetilde{F}(x)$.  Differentiating both sides of this equation $m$ times gives
\begin{displaymath}
2e^x = \widetilde{F}^{(m)}(x)(e^{2x}+1) + \sum_{k=0}^{m-1} {m \choose k} \widetilde{F}^{(k)}(x)2^{m-k}e^{2x},
\end{displaymath}
and evaluating this equation at $x=0$ gives 
\begin{equation}\label{mth-derivative-of-y}
2 = 2\widetilde{F}^{(m)}(0) + \sum_{k=0}^{m-1}{m \choose k}2^{m-k} \widetilde{F}^{(k)}(0) .
\end{equation}
Since the derivative of an even function is odd and the derivative of an odd function is even, $\widetilde{F}^{(2k+1)}(0) = 0$ for all $k$.  Thus for any $n$, evaluating Equation \eqref{mth-derivative-of-y} when $m=2n$ and $m=2n-1$ gives 
\begin{eqnarray*}
2 &=& 2\widetilde{F}^{(2n)}(0) + \sum_{k=0}^{2n-1}{2n \choose k}\widetilde{F}^{(k)}(0)2^{2n-k} \\
&=& 2\widetilde{F}^{(2n)}(0) + \sum_{k=0}^{n-1}{2n \choose 2k} \widetilde{F}^{(2k)}(0) 2^{2n-2k} \text{ and }\\
2 &=& 2\widetilde{F}^{(2n-1)}(0) + \sum_{k=0}^{2n-2}{2n-1 \choose k}\widetilde{F}^{(k)}(0)2^{2n-1-k} \\
&=& 0 + \sum_{k=0}^{n-1}{2n-1 \choose 2k} \widetilde{F}^{(2k)}(0)2^{2n-2k-1}.
\end{eqnarray*}
Thus 
\begin{eqnarray*}
-2\widetilde{F}^{(2n)}(0) &=& \sum_{k=0}^{n-1}{2n \choose 2k} \widetilde{F}^{(2k)}(0) 2^{2n-2k} - \sum_{k=0}^{n-1}{2n-1 \choose 2k} \widetilde{F}^{(2k)}(0)2^{2n-2k-1} \\
&=& \sum_{k=0}^{n-1} 2^{2n-2k-1}\widetilde{F}^{(2k)}(0) \left[2{2n \choose 2k} - {2n-1 \choose 2k}\right],
\end{eqnarray*}
and hence $a_{\mathcal{G}}(2n) = \widetilde{F}^{(2n)}(0)$ for all $n \geq 0$, as desired. 
\end{proof}

\subsection{The simplicial Gromov hyperbolization}

When $a=-1$, the coefficients $b_{\mathcal{G}}(n)$ are $$1,-1,-1,3,3,-17,-17,155,155,\ldots.$$  Once again the sequence $1,1,3,17,155,\ldots$ is the sequence of Genocchi numbers, which are the unsigned nonzero coefficients in the Taylor series expansion of $\widetilde{G}(x) = x \tanh{\frac{x}{2}}= \frac{x^2}{2} - \frac{x^4}{4!} + \frac{3x^6}{6!} + \cdots$  (note the sign change between this case and the simplicial M\"obius band hyperbolization).  

\begin{theorem} \label{gromov-b-are-genocchi}
Let $$\widetilde{G}(x) = -x\frac{e^x-1}{e^x+1}.$$  Then $\widetilde{G}^{(2n-1)}(0) = 0$ and $\widetilde{G}^{(2n)}(0) = b_{\mathcal{G}}(2n-2)$ for all $n \geq 1$. 
\end{theorem}

As in the proof of Theorem \ref{gromov-a-are-euler}, we require several lemmas whose proofs are identical to their cubical counterparts. 

\begin{lemma} \label{gromov-b-even=odd}
For all $n \geq 1$, $$b_{\mathcal{G}}(2n) = b_{\mathcal{G}}(2n-1).$$
\end{lemma}

\begin{lemma} \label{gromov-b-simplification}
For all $n \geq 1$, $$\sum_{k=0}^{2n-1}b_{\mathcal{G}}(k){2n+1 \choose k+1} = 0.$$
\end{lemma}

\begin{lemma}
For all $n \geq 1$, 
\begin{equation} \label{gromov-b-good-recursion}
-2b_{\mathcal{G}}(2n) = -1 + \sum_{k=0}^{n-1}{2n+1 \choose 2k+1}b_{\mathcal{G}}(2k).
\end{equation}
\end{lemma}

Now we are ready to prove Theorem \ref{gromov-b-are-genocchi}. 

\begin{proof}
We prove the claim by induction.  As in Equation \eqref{gromov-G-implicit-diff}, we get $$(e^x+1)\widetilde{G}^{(m)}(x) + \sum_{k=0}^{m-1}\widetilde{G}^{(k)}(x)e^x = (x+m)e^x.$$
Evaluating this equation at $x=0$ when $m = 2n+2$ and $m=2n+1$ gives
\begin{eqnarray*}
2n+2 &=& 2\widetilde{G}^{(2n+2)}(0) + \sum_{k=0}^{2n+1}{2n+2 \choose k}\widetilde{G}^{(k)}(0) \\
&=& 2\widetilde{G}^{(2n+2)}(0) + \sum_{k=1}^{n}{2n+2 \choose 2k}\widetilde{G}^{(2k)}(0) \\
2n+1 &=& 2\widetilde{G}^{(2n+1)}(0) + \sum_{k=0}^{2n}{2n+1 \choose k} \widetilde{G}^{(k)}(0) \\
&=& 0 + \sum_{k=1}^{n}{2n+1 \choose 2k}\widetilde{G}^{(2k)}(0).
\end{eqnarray*}
Subtracting these equations gives 
\begin{eqnarray*}
1 &=& 2\widetilde{G}^{(2n+2)}(0) + \sum_{k=1}^n{2n+1 \choose 2k-1}\widetilde{G}^{(2k)}(0) \\
&=& 2\widetilde{G}^{(2n+2)}(0) + \sum_{k=1}^{n}{2n+1 \choose 2k-1}b_{\mathcal{G}}(2k-2).
\end{eqnarray*}
Thus by Equation \eqref{gromov-b-good-recursion}, $\widetilde{G}^{(2n+2)}(0) = b_{\mathcal{G}}(2n).$
\end{proof}

\bibliography{hyperbolization}

\begin{thebibliography}{10}

\bibitem{Kirby1997}
Problems in low-dimensional topology.
\newblock In Rob Kirby, editor, {\em Geometric topology ({A}thens, {GA},
  1993)}, volume~2 of {\em AMS/IP Stud. Adv. Math.}, pages 35--473. Amer. Math.
  Soc., Providence, RI, 1997.

\bibitem{Adin1996}
R.~Adin.
\newblock A new cubical {$h$}-vector.
\newblock In {\em Proceedings of the 6th {C}onference on {F}ormal {P}ower
  {S}eries and {A}lgebraic {C}ombinatorics ({N}ew {B}runswick, {NJ}, 1994)},
  volume 157, pages 3--14, 1996.

\bibitem{CharneyDavis1995a}
Ruth Charney and Michael Davis.
\newblock The {E}uler characteristic of a nonpositively curved, piecewise
  {E}uclidean manifold.
\newblock {\em Pacific J. Math.}, 171(1):117--137, 1995.

\bibitem{CharneyDavis1995b}
Ruth~M. Charney and Michael~W. Davis.
\newblock Strict hyperbolization.
\newblock {\em Topology}, 34(2):329--350, 1995.

\bibitem{DavisJanuszkiewicz1991}
Michael~W. Davis and Tadeusz Januszkiewicz.
\newblock Hyperbolization of polyhedra.
\newblock {\em J. Differential Geom.}, 34(2):347--388, 1991.

\bibitem{DavisJanuszkiewicWeinberger2001}
Michael~W. Davis, Tadeusz Januszkiewicz, and Shmuel Weinberger.
\newblock Relative hyperbolization and aspherical bordisms: an addendum to
  ``{H}yperbolization of polyhedra'' [{J}.\ {D}ifferential {G}eom.\ {\bf 34}
  (1991), no.\ 2, 347--388; {MR}1131435 (92h:57036)] by {D}avis and
  {J}anuszkiewicz.
\newblock {\em J. Differential Geom.}, 58(3):535--541, 2001.

\bibitem{Gromov1987}
M.~Gromov.
\newblock Hyperbolic groups.
\newblock In {\em Essays in group theory}, volume~8 of {\em Math. Sci. Res.
  Inst. Publ.}, pages 75--263. Springer, New York, 1987.

\bibitem{HershNovik2002}
P.~Hersh and I.~Novik.
\newblock A short simplicial {$h$}-vector and the upper bound theorem.
\newblock {\em Discrete Comput. Geom.}, 28(3):283--289, 2002.

\bibitem{Januszkiewicz1991}
Tadeusz Januszkiewicz.
\newblock Hyperbolizations.
\newblock In {\em Group theory from a geometrical viewpoint ({T}rieste, 1990)},
  pages 464--490. World Sci. Publ., River Edge, NJ, 1991.

\bibitem{Karu2006}
Kalle Karu.
\newblock The {$cd$}-index of fans and posets.
\newblock {\em Compos. Math.}, 142(3):701--718, 2006.

\bibitem{Klee1964}
V.~Klee.
\newblock A combinatorial analogue of {P}oincar\'e's duality theorem.
\newblock {\em Canad. J. Math.}, 16:517--531, 1964.

\bibitem{Munkres1984}
J.~Munkres.
\newblock Topological results in combinatorics.
\newblock {\em Michigan Math. J.}, 31(1):113--128, 1984.

\bibitem{Paulin1991}
Fr{\'e}d{\'e}ric Paulin.
\newblock Constructions of hyperbolic groups via hyperbolizations of polyhedra.
\newblock In {\em Group theory from a geometrical viewpoint ({T}rieste, 1990)},
  pages 313--372. World Sci. Publ., River Edge, NJ, 1991.

\bibitem{Reisner1976}
G.~Reisner.
\newblock Cohen-{M}acaulay quotients of polynomial rings.
\newblock {\em Advances in Math.}, 21(1):30--49, 1976.

\bibitem{OEIS}
N.~J.~A. Sloane.
\newblock {T}he {O}n-{L}ine {E}ncyclopedia of {I}nteger {S}equences, 2012.
\newblock published electronically at \texttt{http://oeis.org}.

\bibitem{Stanley1975}
R.~Stanley.
\newblock The upper bound conjecture and {C}ohen-{M}acaulay rings.
\newblock {\em Studies in Appl. Math.}, 54(2):135--142, 1975.

\bibitem{Stanley-CCA}
R.~Stanley.
\newblock {\em Combinatorics and commutative algebra}, volume~41 of {\em
  Progress in Mathematics}.
\newblock Birkh\"auser Boston Inc., Boston, MA, second edition, 1996.

\bibitem{Stanley1994}
Richard~P. Stanley.
\newblock Flag {$f$}-vectors and the {$cd$}-index.
\newblock {\em Math. Z.}, 216(3):483--499, 1994.

\bibitem{Whitehead1939}
J.H.C. Whitehead.
\newblock {On the asphericity of regions in a 3-sphere.}
\newblock {\em Fundam. Math.}, 32:149--166, 1939.

\end{thebibliography}
\bibliographystyle{plain}

\end{document}